\documentclass[12pt]{article} 
\usepackage[sectionbib]{natbib}
\usepackage{array,epsfig,fancyheadings,rotating}
\usepackage[]{hyperref}  
\usepackage{sectsty, secdot}
\sectionfont{\fontsize{15}{14pt plus.8pt minus .6pt}\selectfont}
\renewcommand{\theequation}{\thesection\arabic{equation}}
\subsectionfont{\fontsize{12}{14pt plus.8pt minus .6pt}\selectfont}

\usepackage{amsmath}
\usepackage{amssymb}
\usepackage{amsfonts}
\usepackage{multirow}
\usepackage{amsthm}
\usepackage{xcolor}
\usepackage{dsfont,bm}
\usepackage{enumitem}
\usepackage{xr}

\newtheorem{theorem}{Theorem}
\newtheorem{lemma}{Lemma}

\newtheorem{proposition}{Proposition}
\newtheorem{condition}{Condition}
\theoremstyle{definition}

\newcommand{\indep}{\raisebox{0.05em}{\rotatebox[origin=c]{90}{$\models$}}}

\begin{document}
	
	
	\renewcommand{\baselinestretch}{1.5}

	\renewcommand{\thefootnote}{}

	
	\fontsize{12}{14pt plus.8pt minus .6pt}\selectfont \vspace{0.8pc}
	\centerline{\large\bf On the Optimality of Nuclear-norm-based}
	\vspace{2pt}
	\centerline{\large\bf Matrix Completion for Problems with Smooth}
	\vspace{2pt} 
	\centerline{\large\bf Non-linear Structure}
	\vspace{.4cm} 
	\centerline{Yunhua Xiang, Tianyu, Zhang, Xu Wang, Ali Shojaie, Noah Simon} 
	\vspace{.4cm} 

	\vspace{.55cm} \fontsize{9}{11.5pt plus.8pt minus.6pt}\selectfont
	
	
	\begin{quotation}
		\noindent {\it Abstract:}
		Originally developed for imputing missing entries in low rank, or approximately low rank matrices, matrix completion has proven widely effective in many problems where there is no reason to assume low-dimensional linear structure in the underlying matrix, as would be imposed by rank constraints. In this manuscript, we build some theoretical intuition for this behavior. We consider matrices which are not necessarily low-rank, but lie in a low-dimensional non-linear manifold. We show that nuclear-norm penalization is still effective for recovering these matrices when observations are missing completely at random. In particular, we give upper bounds on the rate of convergence as a function of the number of rows, columns, and observed entries in the matrix, as well as the smoothness and dimension of the non-linear embedding. We additionally give a minimax lower bound: This lower bound agrees with our upper bound (up to a logarithmic factor), which shows that nuclear-norm penalization is (up to log terms) minimax rate optimal for these problems.

		\vspace{9pt}
		\par
	\end{quotation}\par

	\def\thefigure{\arabic{figure}}
	\def\thetable{\arabic{table}}
	
	\renewcommand{\theequation}{\thesection.\arabic{equation}}

	\fontsize{12}{14pt plus.8pt minus .6pt}\selectfont

	\section{Introduction}\label{sec:intro}
	
	Matrix completion is a framework that has gained popularity in a wide range of machine learning applications, including recommender systems \citep{Koren2009}, system identification \citep{Liu2010}, global positioning \citep{singer2010uniqueness} and natural language processing \citep{Wijaya2017}.
	It is a useful framework for complex prediction problems, where each observation comes with a heterogeneous collection of observed features.	In particular, matrix completion is applied to problems where the object of inference or prediction is a matrix whose rows correspond to observation and columns to variables/features. In many applications, only a subset of entries in this matrix are observed (often with noise), and the goal is to ``complete'' the matrix, filling in estimates of the unobserved entries. This ``completion'' is done by leveraging the known structure in the matrix. The most famous example, which brought matrix completion to prominence, is the Netflix Challenge \citep{Koren2009}, where a small sample of observed ratings for each customer was used to successfully predict future/unobserved movie ratings for Netflix customers.

	More formally, suppose we have an underlying unobserved matrix $M\in\mathbb{R}^{n\times p}$: We then observe a subset of the entries from the noise-contaminated matrix $Y= M +  E$, where $E$ is a matrix of i.i.d. mean zero, finite variance noise variables. Our goal is to recover matrix $M$ from this partially observed, noisy $Y$. This is known as matrix completion. Without any structure on the matrix $M$, recovering the values of $M$ corresponding to unobserved entries is impossible \citep{Laurent2001}. Matrix completion becomes possible if one imposes some constraints on the structure of the underlying matrix: It is most common to assume that $M$ is low rank.
	Directly employing this assumption by e.g., finding the minimum rank completion of $Y$ (or corresponding rank-constrained regression) is unfortunately NP-hard and becomes computationally infeasible for problems involving large matrices \citep{CandesTao2010, Chistov1984ComplexityOQ}.
	Over the last decades, computationally efficient methods using convex optimization have been developed for recovering a low rank matrix from a small number of observations with near-optimal statistical guarantees in primarily noiseless problems \citep{Srebro2004, Recht2011, CandesTao2010, Recht2010}, and when the observed entries are contaminated with noise \citep{CandesPlan2009, koltchinskii2011}. These methods rely on using the nuclear norm of the matrix \citep{fazel2002, jaggi2010simple}, i.e., sum of its singular values, as a convex surrogate for the matrix rank. The low-rank structure leveraged in matrix completion can be thought of as learning a linear embedding of the data in a low-dimensional space.


    In practice, the underlying matrix $M$ may not be low rank. However, we often believe it may still have useful low-dimensional structure. It has thus become popular to learn a low-dimensional non-linear embedding of the data. This idea is used both in matrix completion and more generally for low-dimensional summaries of data. It has been applied in motion recovery \citep{Xia2018}, epigenomics \citep{Schreiber2018}, and health data analytics \citep{wang2015} among other areas. To recover these embeddings, Reproducing Kernel Hilbert Space (RKHS) methods \citep{Fan2018a}, nearest neighbor methods \citep{li2019nearest}, and deep learning methods like autoencoders and neural-network-based variational frameworks \citep{Fan2018b, yu2013embedding, jiang2016variational} have been used.

	Additionally, there has been strong empirical evidence that matrix completion methods based on nuclear norm penalization perform well even in scenarios where any low dimensional structure is likely non-linear. As these methods were developed for linear low rank structure, this is, at first glance, a bit surprising. There has been some work giving theoretical justification for these empirical results \citep{chatterjee2015matrix, udell2019big}. In particular, they note that in the presence of some types of non-linear low-dimensional structure in $M$, nuclear norm-based matrix completion methods can still consistently estimate $M$. These work additionally gives some non-stochastic approximation error results. However, optimality of the statistical perform of nuclear-norm-based matrix completion is not considered to the best of our knowledge.


		In this manuscript, we delve further into the performance of matrix completion for $M$ with low-dimensional, non-linear structure. In particular, we consider $M$ with rows that can be embedded in a low-dimensional smooth manifold. We then (i) show that nuclear norm-based matrix completion can consistently estimate $M$; (ii) characterize the rate at which the reconstruction error converges to $0$ as a function of the size of the matrix, number of observed entries, and smoothness and dimension of the underlying manifold; and (iii) prove that, up to a log term, this rate cannot by improved upon by any method; that is, our upper bound is actually the minimax rate optimal for reconstruction error in this problem.  Furthermore, our error bounds (and our techniques) also relate the matrix completion problem clearly to more classical non-parametric estimation: Our reconstruction error bounds parallel the minimax rate of mean squared error (MSE) in the nonparametric regression setting. Results (ii) and (iii), we believe, are novel.
		

		Our experiments on synthetic data corroborate our theoretical findings. In particular, they suggest that the finite sample empirical performance of matrix completion in non-linear low rank embeddings is consistent with the asymptotic theoretical error bounds. These empirical results also corroborate the claim that better performance is achieved when the embedding of the underlying matrix $M$ lies in a smoother manifold.

		\section{Methods}\label{sec:setup}

		\subsection{Problem setup}
		We start by giving some notation. We use upper case letters to represent matrices and lower case letters to represent scalars. The trace inner product of any two matrices, $M, B\in \mathbb R^{n\times p},\ n,p \in \mathbb{Z}^+$, is $\langle M, B\rangle = \operatorname{tr}(M^TB)$. The element-wise infinity norm of $M\in \mathbb R^{n\times p}$ is defined by $\|M\|_\infty= \max_{1\le i \le n, 1\le j \le p}|m_{ij}|$ where $m_{ij}$ denotes the $(i,j)$-th entry of $M$. We also denote the Frobenius norm of matrix $M$ as $\|M\|_F = \sqrt{\sum_{i=1}^{n}\sum_{j=1}^{p}m_{ij}^2}$.

		In the general matrix completion problem, we randomly observe some of the entries from a matrix $M\in\mathbb{R}^{n\times p}$; the observed entries may also be contaminated with error. To support our later theoretical derivations, we will describe this process in terms of a set of mask matrices $X_t\in \mathbb{R}^{n\times p}$ and observed values $y_t \in\mathbb{R}$. Each $X_t$ is a matrix with a single $1$ whose position is indexed by $t$ and all other entries are equal to $0$ as follows:
		\begin{equation}\label{eq::mask_mat}
			X_t = \begin{pmatrix} 
				0 & 0& \cdots & 0 & \cdots & 0 \\
				\vdots &  & \vdots &   & \vdots\\
				0 &0  & \cdots & 1 &\cdots & 0\\
				\vdots &  & \vdots &   & \vdots\\
				0 &0 & \cdots & 0 & \cdots & 0 \\
			\end{pmatrix}_{n\times p}.
		\end{equation}
		The collection of matrices $X_t$ fall in the set $\mathcal{X} = \{e_n(i)e_p(j)^T, \textrm{ for all }i=1,\ldots,n \textrm{ and  }j=1,\ldots,p\}$, where $e_n(i)\in \mathbb R^n$ is the basis vector consisting of all zeros except for a single 1 at $i$th entry. In this formulation, $X_t$ indicates the location in $M$ where $y_t$ is drawn from. That is, for $X_t = e_n(i)e_p(j)^T \in \mathcal{X}$, $\langle X_t, M\rangle = m_{ij}$. 
		
		Now, we can frame the matrix completion problem as follows: Suppose we have $N$ pairs of observations $(X_t,y_t)$,  $t=1,\ldots, N$, that satisfy
		\begin{equation}\label{eq::data_model}
			y_t = \langle X_t, M\rangle + \xi_t,
		\end{equation}
		where $\xi_t$ are i.i.d random errors distributed $N(0,\sigma^2)$, $M \in \mathbb{R}^{n\times p}$ is the underlying true matrix to be recovered, and $y_t\in \mathbb{R}$ are observed values. The observed matrix can be written as $Y = \sum_{t=1}^{N}y_tX_t$ where $N$ is the number of observed entries. We assume that $X_t$ is uniformly sampled at random from $\mathcal{X}$ \citep{koltchinskii2011}, i.e. $X_t\sim \Pi$, and the probability that the $(i,j)$th entry of $X_t$ equals to 1 is $\pi_{ij} = \operatorname{P}(X_t = e_i(n)e_j(p)^T) = \frac{1}{np}$ for $1 \le i \le n,  1 \le j \le p$. This is essentially a missing completely at random (MCAR) assumption. 
		
				
		
		The goal is to recover $M$ given pairs $(X_t, y_t)$, $t = 1,2,...,N$, and we are generally interested in the setting where $N\ll np$. To solve this problem, existing methods often assume that $M$ has low rank (or approximately low rank), i.e. $M \simeq UV^T$ with $U \in \mathbb{R}^{n\times r}$ and $V \in \mathbb{R}^{p \times r}$ for some integer $r \ll \min(n,p)$.  In contrast to this low rank assumption, this paper studies the problem where $M$ is not necessarily low-rank but generated from a low-dimensional non-linear manifold. This notion is formalized in the next section.

		\subsection{Non-linearly Embeddable Matrices}
		We begin by formalizing what we mean by ``low-dimensional non-linear structure''. Consider a matrix $M$, a positive integer $K$, and a function class  $\mathcal{F} \subset \mathcal{L}^2\left(\mathbb{R}^K\right)$. We say $M$ is $\mathcal{F}$-embeddable if there exist functions $f_j\in\mathcal{F}: \mathbb{R}^K \to \mathbb{R}, \ j=1,\ldots,p$, and a matrix $\bm\Theta\in\mathbb{R}^{n\times K}$ such that
		\begin{equation}\label{eq::mat_gen}
			m_{ij} = f_j\left(\bm\theta_{i,\cdot}\right), i=1,\ldots,n, j=1,\ldots,p,
		\end{equation}
		where $m_{ij}$ is the ($i,j$) entry of $M$ and $\bm\Theta \in \mathbb{R}^{n\times K}$ is a matrix (with $\bm\theta_{i,\cdot}$ indicating its $i$th row vector). Here, $\bm\Theta$ gives an embedding of our observations from its original $p$-dimensional space into a $K$-dimensional space ($K \le p$). 
	The set of functions $\{f_j\}_{j=1}^p \subset \mathcal{F}$ identifies how to map our embedding in $\mathbb{R}^K$ back to $\mathbb{R}^p$. 

		In classical matrix completion setting, where we assume $M$ is low-rank, nuclear norm penalized empirical risk minimization is often used to estimate $M$ \citep{argyriou2008convex, candes2010matrix, negahban2011estimation}; more specifically, the estimator is obtained by,
		\begin{equation}\label{eq::est_proc_old}
		 \arg\min_{M}\left\{N^{-1}\sum_{t=1}^N(y_t - \langle X_t, M\rangle)^2 + \lambda \|M\|_*\right\},
		\end{equation}
		where $\lambda$ is a regularization parameter which is used to balance the trade-off between fitting the unknown matrix using least squares and minimizing the nuclear norm $\|M\|_*$. This ``matrix lasso'' is known to have strong theoretical properties when $M$ is low rank \citep{argyriou2008convex,candes2010matrix,negahban2011estimation,cai2016matrix}. However, in our scenario, $M$ likely does not have low rank and previous work does not fully explain the effectiveness of the estimate from \eqref{eq::est_proc_old} in this setting.

		
		While the estimator in \eqref{eq::est_proc_old} is simple and quite well known, it fails to exploit knowledge of the sampling scheme (which is often known or at least assumed to be known). To use the assumption that the mask matrices $\{X_t\}_{t=1}^N$ are i.i.d. uniformly sampled from $\mathcal{X}$, we study a slight modification to \eqref{eq::est_proc_old} described in \citet{koltchinskii2011}:
		\begin{equation}\label{eq::est_proc}
			\begin{aligned}
				\widehat M & \leftarrow \arg\min_{M} \left\{\frac{1}{np}\|M\|_F^2 - \left\langle \frac{2}{N} \sum_{t=1}^{N} y_t X_t, M \right\rangle + \lambda \|M\|_*\right\}
			\end{aligned}
		\end{equation}
		
		After some simple manipulation, \eqref{eq::est_proc} can be further reduced to minimizing
		\[
		\frac{1}{np}\|M-R\|_F^2 +\lambda \|M\|_* .
		\] 
		where $R = \frac{np}{N}\sum_{t=1}^{N}y_tX_t = \frac{np}{N}Y$. Thus, $\widehat{M}$, the solution to \eqref{eq::est_proc}, is merely a singular-value soft-thresholding estimator:
		\begin{equation}\label{eq::hat_M}
			\widehat M = \sum_{j=1}^{\operatorname{rank}(R)}(\Lambda_j(R)-\lambda np/2)_+u_j(R)v_j(R)^T,
		\end{equation}
		where $\Lambda_j(R)$ are the singular values and $u_j(R)$, $v_j(R)$ are the left and right singular vectors of $R$ such that $R = \sum_{j=1}^{\operatorname{rank}(R)}\Lambda_j(R)u_j(R)v_j(R)^T$. \citet{koltchinskii2011} established the rate optimality of this estimator with respect to Frobenius-norm loss when $M$ is low rank. In this paper, we aim to ultimately claim that $\widehat M$ in~\eqref{eq::est_proc} is still a consistent and rate optimal estimator of $M$ in the case that $M$ is non-linearly embeddable, as long as $K$ is small and the function class $\mathcal F$ is sufficiently smooth. 

		\subsection{Approximation of Embeddable Matrices}\label{sec::approximation}
		Our goal is to show that the estimator obtained by \eqref{eq::est_proc} is consistent for the true underlying matrix $M$ with respect to Frobenius-norm loss (and characterize the convergence rate), when $M$ is non-linearly embeddable. To this end, we first show that $M$ can be well approximated by a series of matrices with low (and only slowly growing) rank as long as the function class $\mathcal F$ is sufficiently smooth. More specifically, we will need the following condition for the function class $\mathcal{F}$.

		
	 \begin{condition}\label{cond::approx}
			Given a function class $\mathcal{F}$, let $C_0$ denote a fixed positive number. Suppose that for any $\epsilon > 0$, there exists a finite set of functions 
			$\mathcal{F}_{\epsilon} = \left\{\psi_1, \psi_2, \ldots, \psi_{J(\epsilon)}\right\} \subset \mathcal{F}$, such that
			\begin{equation}\label{eq:bound}
			   \left\|\psi\right\|_{\infty} \leq C_0,\quad\text{for all }\psi\in\mathcal{F}_{\epsilon},
			\end{equation}
			and
			\begin{equation}\label{eq:appx}
				\max_{f\in\mathcal{F}}\min_{\left\|\beta\right\|_2^2 \leq C_0} \left\|f - \sum_{l=1}^{J(\epsilon)}\beta_l \psi_l\right\|_{\infty} \leq \epsilon.
			\end{equation}
			For each $\epsilon$, we denote by $\mathcal{F}^{*}_{\epsilon}$ a set of minimal cardinality such that \eqref{eq:bound} and \eqref{eq:appx} hold. We let $J^{*}(\epsilon)$ denote the cardinality of $\mathcal{F}^{*}_{\epsilon}$.
		\end{condition}

		For a function class $\mathcal{F}$, Condition~\ref{cond::approx} characterizes the minimal number of basis functions needed to uniformly approximate functions in $\mathcal{F}$ up to precision $\epsilon$. In Section~\ref{sec::theory}, we shall apply this condition to $K$-dimensional, $L$-th order differentiable functions, and show how this number scales as a function of $\epsilon$. 

		Based on the above condition, we can establish the existence of an approximation matrix which is sufficiently close to the true matrix $M$ and has a bounded nuclear norm. 
		
		\begin{lemma} \label{lem::approx}
			Suppose matrix $M\in\mathbb{R}^{n\times p}$ is $\mathcal{F}$-embeddable, and $\mathcal{F}$ satisfies Condition~\ref{cond::approx}. Then, for any $\epsilon>0$, there exists a matrix $M^\epsilon$ satisfying  $\operatorname{rank}(M^\epsilon) = J^*(\epsilon) \le \min(n,p)$ such that 
			\begin{equation}\label{eq::approx_mat}
				\left\|M^{\epsilon} - M\right\|_{\infty} \leq \epsilon. 
			\end{equation}
			
			Furthermore, the nuclear norm of $M^\epsilon$ is bounded: There exists $C_1>0$ (independent of $\epsilon$) such that
			\begin{equation}\label{eq::bound_nuclear}
				\frac{1}{\sqrt{np}}\left\|M^{\epsilon}\right\|_{*} \leq C_1 J^*(\epsilon).
			\end{equation}
		\end{lemma}
		
		The proof is given in Appendix~\ref{proof_lem_approx}. Note, for the $\mathcal{F}$ we consider later (restricted to smooth functions) we will show that $J^{*}(\epsilon) << \operatorname{min}(n,p)$. This parallels results in classical non-parametric regression where many function-spaces considered can be approximated uniformly with small error by linear combinations of relatively few basis functions \citep{tsybakov2009introduction}.
		

		\section{Consistency}\label{sec::theory}

		Using Lemma~\ref{cond::approx}, it is relatively straightforward to evaluate the performance of our estimator $\widehat M$ in \eqref{eq::est_proc}. The performance metric simplest to theoretically analyze is $N^{-1}\sum_{i=1}^{N}\left\langle X_i, \widehat{M} - M\right\rangle^2$. However, this criterion only evaluates the prediction error on the \emph{observed} entries. This is unsatisfying as our ultimate goal is to recover the entire matrix. 
		Thus, we instead aim to evaluate the performance of $\widehat{M}$ based on the metric  $\frac{1}{np} \|\widehat M - M\|_F^2$. The following result gives an upper bound for the performance of our estimator $\widehat{M}$ in this metric.

		\begin{theorem}\label{thm::upper_bound}
			Suppose we observe N pairs $\{(y_t, X_t)\}_{t=1}^N$ satisfying data generating model \eqref{eq::data_model} where $X_t$ are i.i.d. uniformly sampled from $\mathcal{X}$.  Assume the true matrix $M \in \mathbb{R}^{n\times p}$ is $\mathcal{F}$-embeddable where $\mathcal{F}$ satisfies Condition \ref{cond::approx}. Further suppose that $N \ge (n\wedge p)\log^2(n+p)$.  Then there exists a constant $C_2 >0$ (that only depends on $\sigma$ and $ \|M\|_\infty$) such that if we define the regularization parameter $\lambda$ by
			\[\lambda =  C_2\sqrt{\frac{\log (n+p)}{N(n\wedge p)}},
			\]
			then, with probability at least $1-2(n+p)^{-1}$, the completion error of $\widehat M$ in \eqref{eq::hat_M} is bounded by
			\begin{equation}\label{eq::generalbound}
				\frac{1}{np}\left\|\widehat M-M\right\|_F^2 \le
				 C_2^2\left(\frac{1+\sqrt{2}}{2}\right)^2\frac{(n\vee p)\log(n+p)}{N}J^*(\epsilon) 
				+ \epsilon^2,
			\end{equation}
			for any $\epsilon>0$. Here, $J^*(\epsilon)$ is the rank of the approximation matrix $M^\epsilon$ with $\|M - M^\epsilon\|_\infty \le \epsilon$, which corresponds to the minimal cardinality of $\mathcal{F}^* \subset \mathcal{F}$ satisfying Condition~\ref{cond::approx}.
		\end{theorem}

		The upper bound in Theorem~\ref{thm::upper_bound} can be established by extending the results from \citet{koltchinskii2011}. The details of the proof are given in the Appendix~\ref{proof_upper_bound}. The two terms on the right-hand-side of \eqref{eq::generalbound} clarify the trade-off between the approximation error, $\epsilon$, and the cardinality of the minimal linear approximation set $\mathcal{F}^*$, $J^*(\epsilon)$. Our upper bound is consistent with the results in \citet{koltchinskii2011}, where the error is decomposed into a misspecification error $(\epsilon^2)$ and a prediction error. Usually, when there is no misspecification, i.e., the true matrix $M$ is low rank, the prediction error is linearly related to the rank of $M$ \citep{candes2011tight, klopp2014noisy, cai2016matrix}. In our scenario, where the low-rank assumption is violated, the prediction error in \eqref{eq::generalbound} is linearly related to the rank of the approximation matrix.

		Ideas similar to this occur in more traditional non-parametric estimation problems. For example, when using projection estimators in H\"older and Sobolev spaces, one of the main rate-optimal estimation approaches requires a truncated basis to be selected for projection that will grow with the sample size $N$ \citep{tsybakov2008introduction}. However, in those examples, the number of basis vectors is a tuning parameter in the algorithm, and the set of basis functions must be selected in advance. Here, both the set of basis functions and the truncation level are rather just theoretical tools for analyzing the algorithm performance. In employing matrix completion, the analyst only needs to select $\lambda$.
		
		We note that $N \ge (n\wedge p)\log^2(n+p)$ in the above Theorem~\ref{thm::upper_bound} is a quite weak condition on the number of observations: $N$ could satisfy this and still be far less than $np$. For the results of the latent space model in \cite{chatterjee2015matrix}, they require at least $O\left( n^{\frac{2(K+1)}{K+2}}\right)$ entries to be observed out of $n^2$ entries to guarantee the consistency for recovering an $n\times n$ matrix. This implies that one needs to observe $O\left( n^{\frac{K}{K+2}}\right)$ entries out of $n$ in each row, as compared to our much weaker requirement of $O\left(\log^2(n)\right)$ per row.


		We now specialize our results to matrices that are $\mathcal{F}$-embeddable for $\mathcal{F}$ containing functions with bounded derivatives. This is a natural class of functions to work with (though one could alternatively work in a multivariate Sobolev or H\"older space).

	\begin{condition}\label{cond::bound_derivative}
				$M$ is $\mathcal{F}$-embeddable, where $\mathcal{F}$ contains functions with uniformly bounded $L$-th order mixed partials (for some fixed $L>0$). More formally, define $\mathcal{F}(L,\gamma,K)$, for $L,K \geq 1$ as the set of $L$-th order differentiable functions from $\mathbb{R}_{[0,1]}^K$ to $\mathbb{R}$ satisfying
			\begin{equation}\label{eq::bound_deriv}
				\left|\frac{\partial^L}{\partial x_1^{L_1}\cdots x_K^{L_K}} f(\mathbf{x}) \bigg\vert_{\mathbf{x}=\mathbf{x^0}}\right| \leq \gamma,
			\end{equation}
			for all $\mathbf x^0  = (x_1^0,\ldots,x_K^0)\in \mathbb{R}_{[0,1]}^K \subset \mathbb{R}^K$ and all integers $L_1,\ldots,L_K$ satisfying $L_1+\cdots+L_K=L$. Now, additionally define the set
			\begin{equation}\label{eq::mat_class}
				\begin{aligned}
					\mathcal{M}(L,\gamma,K) &= \{ M \in \mathbb{R}^{n\times p}\,\mid \, m_{ij} = f_j(\bm\theta_{i,\cdot}),\\ &\textrm{ with } f_j\in\mathcal{F}(L,\gamma,K),\ j\leq p,\text{ and }\bm\theta_{i,\cdot}\in\mathbb{R}_{[0,1]}^K,\ i \leq n\}
				\end{aligned}
			\end{equation}
			This is the set of $F(L,\gamma,K)$ embeddable matrices, where the embedding lives in a compact space (for convenience we use the $\ell_{\infty}$  ball). Our formal condition here is that $M\in 	\mathcal{M}(L,\gamma,K)$.
		\end{condition}

		{\bf Remark.} In the above condition, we will often suppress the dependence on $\gamma$, and write $\mathcal{M}(L,K)$ and $\mathcal{F}(L,K)$. This is because $\gamma$ does not affect the convergence rate of our estimator. Additionally, here we specify the domain of the embeddings to be $[0,1]^K$ for ease of exposition. This is actually general as we could rescale any compactly supported embedding to live in this interval.

		Condition~\ref{cond::bound_derivative} imposes an additional constraint on our embedding: The underlying manifold on which our matrix lives should be smooth. Here smoothness is characterized by a number of bounded derivatives. As we will see, this function class engages well with Condition~\ref{cond::approx} in the sense that we are able to characterize $J^{*}(\epsilon)$ for the function class $\mathcal{F}(L,K)$. This is essentially a multivariate H\"older class, which has been widely used in the area of non-parametric estimation \citep{tsybakov2008introduction}. One could alternatively look at this as a multivariate Sobolev class under the sup-norm, $W^{L,\infty}(\mathbb{R}^K)$.

		The following lemma gives the number of basis elements needed to linearly approximate a matrix satisfying the above condition, with bounded approximation error $\epsilon$.
		
		\begin{lemma}\label{lem::J_star_bound}
			For the function class $\mathcal{F}(L,K)$ described in Condition~\ref{cond::bound_derivative}, we have that Condition~\ref{cond::approx} is satisfied with $J^*(\epsilon) = O\left(\epsilon^{-K/L}\right)$.
		\end{lemma}
		
		The proof of this lemma is given in Appendix~\ref{proof_J_star_bound}. Now, we can establish the final convergence result for smoothly embeddable matrices.

		\begin{theorem}\label{thm::upper_bound2}
			Under the same scenario and assumptions as in Theorem \ref{thm::upper_bound}, assume further the $\mathcal{F}(K,L)$-embeddable matrix $M$ satisfies Condition~\ref{cond::bound_derivative} for a given $L$ and $K$. Then, the upper bound \eqref{eq::generalbound} is optimized at $\epsilon = \left(\frac{(n\vee p)\log(n+p)}{N}\right)^{\frac{L}{2L+K}}$, resulting in 
			\begin{equation}\label{eq::upper_bound}
				\begin{aligned}
					\frac{1}{np}\left\|\widehat M-M\right\|_F^2 &=O_P\left(\left[\frac{(n\vee p)\log(n+p)}{N}\right]^{\frac{2L}{2L+K}}\right). 
				\end{aligned}
			\end{equation}
		\end{theorem}

		The proof is given in Appendix~\ref{proof_upper_bound2}. This upper bound of the convergence rate of the MSE of $\widehat M$ is only based on the dimensions $n$ and $p$ of matrix $M$, the total number of observations $N$, as well as the degree of smoothness $L$ and dimension of the embedding $K$. Previous work that assumed $M$ was low-rank generally gave a rate of the form $N^{-1}\operatorname{(n\vee p) rank(M)\log(n+p)}$ \citep{bach2008consistency, klopp2014noisy, van2016estimation}. In contrast, our upper bound does not rely on the rank of $M$. Instead, the role of $\operatorname{rank}(M)$ is replaced by $L$, and $K$. This result reaffirms that the standard matrix completion estimator based on nuclear norm minimization is consistent for matrices with low-dimensional non-linear structure. Perhaps more importantly, it also shows how the convergence rate depends on the degree of smoothness, and dimension of the manifold. This can be seen in the exponent on the RHS of \eqref{eq::upper_bound}: $2L/(2L+K)$. Increasing the degree of smoothness moves this exponent towards $1$; increasing the dimension moves the exponent towards $0$. This is analogous to more standard non-parametric regression problems in smooth hypothesis spaces where the minimax convergence rate for MSE looks analagous \citep{tsybakov2008introduction}.

		\section{Minimax Lower Bound}\label{sec::minimax}

		In this section, we use information-theoretical methods to establish a lower bound on the estimation error for completing \emph{non-linearly embeddable} matrices with uniformly sampled at random entries when the latent embedding $\bm\Theta$ is $K$-dimensional and satisfies Condition~\ref{cond::bound_derivative}. The rate we find in the lower bound matches the rate obtained by nuclear norm penalization in Theorem~\ref{thm::upper_bound2} up to a log-term. Thus our upper bound is sharp (up to a logarithmic factor), and, the nuclear-norm penalization based estimator given in \eqref{eq::est_proc} is rate-optimal (up to polylog) for this problem. 

		
	    To derive the lower bound, we consider the underlying matrices $M \in \mathcal{M}(L,\gamma,K)$ as defined in \eqref{eq::mat_class}, i.e., matrices that live in $L$-th order smooth, $K$ dimensional manifolds. Let $\mathbb{P}_M$ denote the probability distribution of the observations $\{(y_t,X_t)\}_{t=1}^N$ generated by model \eqref{eq::data_model} with $\operatorname{E}(y_t|X_t) = \langle X_t, M\rangle$. We give a minimax lower bound of the $\|\cdot\|_F^2$-risk for estimating $M$ in the following result.


		\begin{theorem}\label{thm::lower_bound}
			For any given $L\geq 1$, $\gamma>0$ and $K\geq 1$, let $\kappa:=n/p$. Then, for some constant $A>0$ that depends on $K, L, \gamma, \sigma^2$ and $\kappa$, the minimax risk for estimating $M$ satisfies
			\begin{equation}\label{eq::lower_bound}
				\inf_{\hat M}\sup_{M \in \mathcal{M}(L,\gamma,K)} \mathbb{P}_M\left(\frac{1}{np}\left\|\widehat M - M\right\|_F^2 > A \left(\frac{n\vee p}{N}\right)^{\frac{2L}{2L+K}} \right) \ge 1/2,
			\end{equation}
			when $c_0^{-\frac{2L+K}{K}}(n\vee p) \le N \le c_0^{-\frac{2L+K}{K}} 0.48^{2L+K} (n\vee p)n^{\frac{2L+K}{K}}$ for some constant $c_0$ which depends on $K, L, \gamma, \sigma^2$ and $\kappa$.
		\end{theorem}

		The proof is given in the Appendix~\ref{proof_lower_bound}. Comparing Theorem \ref{thm::lower_bound} to Theorem \ref{thm::upper_bound2}, we see that the lower bound matches the upper bound \eqref{eq::upper_bound} up to a logarithmic factor. This shows that the estimator given by \eqref{eq::est_proc} is actually an optimal estimator (up to a log term) for this non-linear low-dimensional matrix completion regime.
		
		We note that the requirement $N = O\left((n\vee p)n^{\frac{2L+K}{K}}\right)$ in Theorem \ref{thm::lower_bound} is a bit unusal. It comes from a technical constraint in our proof, required to construct a suitably large packing set. This may just be an artifact of our proof technique, and not innate to the problem.  Recall that the upper bound holds as long as $N \ge (n\vee p) \log^2(n+p)$, so there is a large regime where the assumption required for our upper and lower bounds overlap.

		.


		\section{Simulation Study}\label{sec::sims}

		In this section, we empirically evaluate the effectiveness of matrix completion using the soft-thresholding estimator $\widehat M$ in \eqref{eq::hat_M} for noisy incomplete matrices which are generated from low-dimensional non-linear embeddings. (These matrices are full rank, even though they are generated from low-dimensional non-linear embeddings). Here, we only show the case of univariate embedding ($K=1$) and aim to empirically evaluate how the Frobenius error $\frac{1}{np}\left\|\widehat M-M\right\|_F^2$ changes with the dimension ($n$) when $n=p$. We examine scenarios where the non-linear embeddings are of different orders of smoothness.

		The underlying matrices are generated as described in \eqref{eq::mat_gen}: $m_{ij} = f_j(\bm\theta_{i,\cdot})$ for $i=1,\ldots,n$ and $j = i,\ldots,p$. In particular, to make sure that Conditions \ref{cond::approx} and \ref{cond::bound_derivative} are satisfied, we generate $f_j$ as
		\[
		f_j(x) = \sum_{b=1}^{\infty}\beta_b\psi_b(x),
		\]
		where $\psi_b(x)$ are orthonormal bases in $L_2[0,1]$ defined by:
		\[
		\begin{aligned}
			\psi_1(x) & = 1,\\
			\psi_{2b}(x) & = \sqrt{2}\cos(2\pi b x),\\
			\psi_{2b+1}(x) & = \sqrt{2}\sin(2\pi b x).
		\end{aligned}
		\] 
		Meanwhile, to set up the order of smoothness $L$ and make sure that $\beta_b\psi_b(x)$ vanishes with $b$, we sample the coefficients $\beta_b$ from a uniform distribution:
		\[
		\beta_b \sim_{i.i.d} U\left[-b^{-(L+1)}, b^{-(L+1)}\right], \quad b=1,2,\ldots .
		\]
		In this way, we can guarantee that $\sum_{b=1}^\infty b^{2L}\beta_j^2 < \infty$. Thus, $f_j$ is a function whose $L$th order derivative is $O_p(1)$.

		In this simulation, for computational reasons, we actually use only the first $100$ basis vectors $f_j(x) = \sum_{b=1}^{100}\beta_b\psi_b(x)$. The underlying embeddings $\bm\theta_{i,\cdot} \in \mathbb{R}$ are also i.i.d. sampled from a uniform distribution $U(0,1)$ for $i=1,\ldots,n$. We set the missingness rate to $\nu = 0.3$: The total number of observed entries is $N=(1-\nu) np$. The observed entries are $y_t = \langle X_t, M\rangle + \xi_t$, where $X_t$ are uniformly sampled from $\mathcal{X}$ and the error terms are independently Gaussian distributed $\xi_t \sim_{i.i.d.} N(0,1)$. We generate random data sets $\{(y_t,X_t)\}_{t=1}^N$ of size $n \in \{500, 1000, 2000, 3000, 5000\}$ and estimate $M$. We run 100 simulations for each size. To select $\lambda$, instead of using cross-validation, here we consider an oracle procedure: For each simulation, we estimate the MSE for a set of $\lambda$ values and select the $\lambda$ that minimizes the MSE. We report this MSE of the estimated matrix $\widehat{M}$ and the corresponding $\lambda$.

		\begin{figure}[t!]
			\centering
			\includegraphics[width=\textwidth]{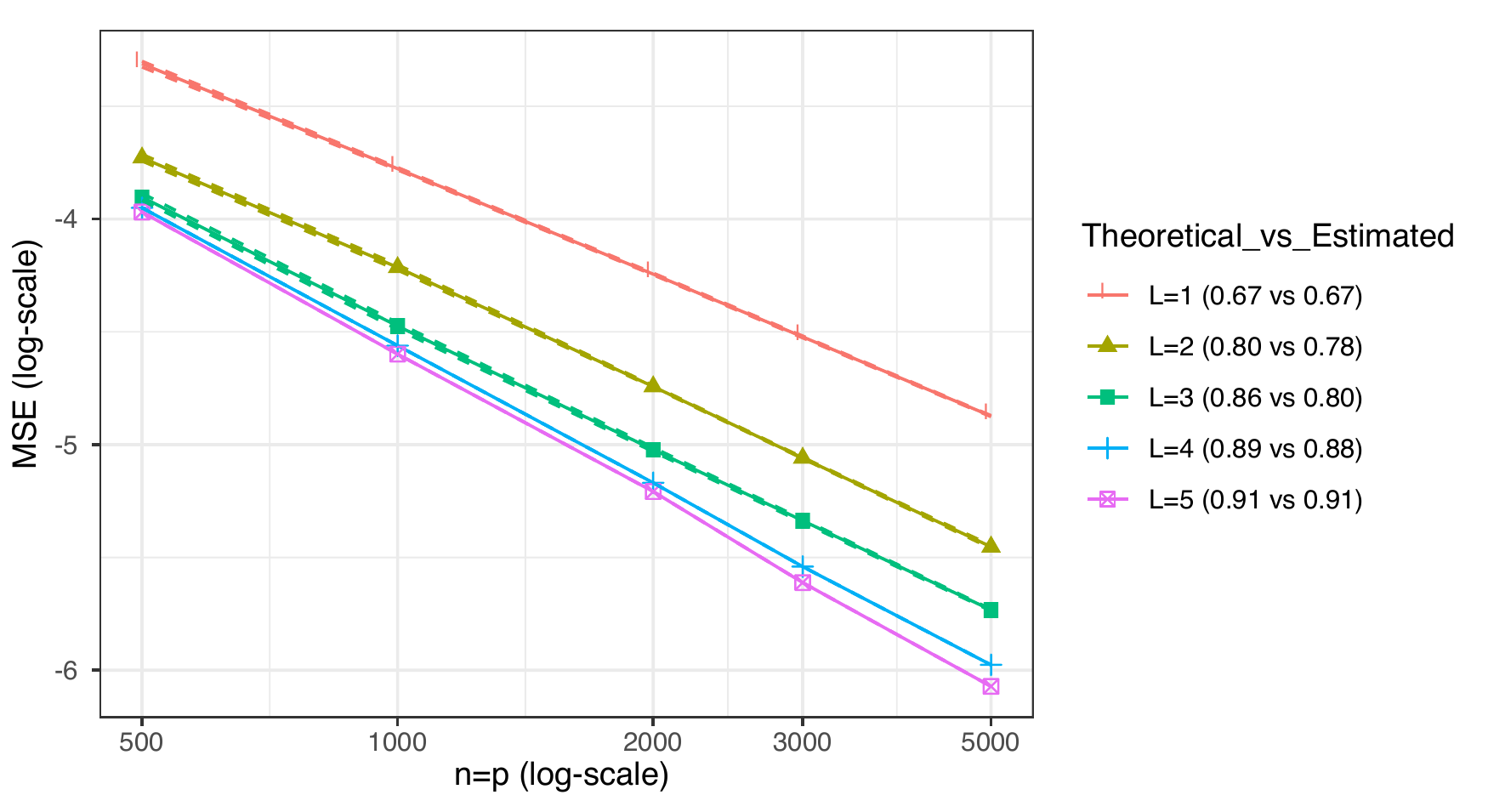}
			\caption{Theoretical rate vs. empirical rate (in log scale) of the mean squared errors as a function of sample size. The underlying matrices $M$ are generated by $f$ with different orders ($L$) of smoothness. The low-rank embedding is one-dimensional ($K=1$). We regress log(MSE) on log($n$), and compare the theoretical slopes (left) with the empirical slopes (right). For each smoothness level, $L$, we also obtain the 95\% confidence regions using bootstrap (dash lines).
			}\label{fig::simulation}
		\end{figure}
		
		Figure \ref{fig::simulation} shows the results of estimating $M$ generated by non-linear embeddings with different orders of smoothness, $L$. Since $N = (1-\nu)np$, the convergence rate in \eqref{eq::upper_bound} reduces to $O_P\left(\left[\log(2n)/n\right]^{\frac{2L}{2L+K}}\right)$. The log term inside is negligible as $n$ increases. Hence, if we regress log(MSE) on $\log(n)$, the absolute value of slope should be roughly about $2L/(2L+1)$ ($K=1$ in this simulation). We increase the order of smoothness of $f$ from $L=1$ to $L=5$. For these values of $L$, the expected absolute value of the slope should be 0.67, 0.80, 0.86, 0.89, and 0.91. The rates from our simulations are respectively 0.67, 0.78, 0.80, 0.88, and 0.91. There is generally strong agreement between theoretical and empirical results except for the setting of $L=3$. We hypothesize that this is due to finite sample issues.

		\section{Discussion}\label{sec::discussion}

		Nuclear-norm based matrix completion methods were originally developed for scenarios where the underlying mean matrix has low rank. In this manuscript, we present theoretical results to explain the effectiveness of matrix completion in applications where the underlying mean matrix is not low rank, but instead lives in a low-dimensional smooth manifold.
		
		Our results show that, in such scenarios, nuclear-norm regularization can still result in a procedure that is minimax rate optimal (up to a log factor) for recovering the underlying mean matrix. In particular, we give upper bounds on the rate of convergence as a function of the number of rows, columns, and observed entries in the matrix, as well as the smoothness, and dimension of the embeddings. We additionally give matching minimax lower bounds (up to a logarithmic factor) for this problem. These bounds appear analogous to the minimax rate in the case of standard non-parametric regression.
		
		Our theoretical results relate the error bounds to the smoothness and dimension of the non-linear embedding; however, the technical proof does not provide a way to figure out the explicit form of the hidden embeddings, which may be interesting in practice, e.g., for dimension reduction.  Modifying the original matrix completion method in order to estimate the hidden embeddings may be an important direction of future research.

\newpage

\appendix
\raggedbottom\sloppy

\fontsize{12}{14pt plus.8pt minus .6pt}\selectfont \vspace{0.8pc}
\centerline{\large\bf Supplementary Materials: On the Optimality of}
\vspace{2pt} 
\centerline{\large\bf Nuclear-norm-based Matrix Completion for Problems }
\vspace{2pt} 
\centerline{\large\bf with Smooth Non-linear Structure}

\def\thefigure{\arabic{figure}}
\def\thetable{\arabic{table}}

\renewcommand{\theequation}{\thesection.\arabic{equation}}

\fontsize{12}{14pt plus.8pt minus .6pt}\selectfont

	\section{Proof of Lemma~\ref{lem::approx}} \label{proof_lem_approx}

We begin by giving a proof of Lemma~\ref{lem::approx}:

\begin{proof}
	
	Recall that $M$ is $\mathcal{F}$-embeddable and  $\mathcal{F}$ satisfies Condition \ref{cond::approx}. Thus, the entries of $M$ are generated by $m_{ij} = f_j(\bm{\theta}_{i,\cdot})$.
	Consider arbitrary $\epsilon>0$. Then there is some fixed $C_0>0$, and a collection of functions $\mathcal{F}^{*}_{\epsilon}= \left\{\tilde\psi_1, \tilde\psi_2, \ldots, \tilde\psi_{J^*(\epsilon)}\right\} \subset \mathcal{F}$ that give the finite set of minimal cardinality $J^{*}(\epsilon)$, with the property that
	$
	\operatorname{max}_{f\in\mathcal{F}}\operatorname{min}_{\left\|\beta\right\|_2^2 \leq C_0} \left\|f - \sum_{l=1}^{J^*(\epsilon)}\beta_l \tilde\psi_l\right\|_{\infty} \leq \epsilon.
	$
	and $\|\tilde{\psi}_l\|_\infty \le C_0$.
	For any given $f\in\mathcal{F}$, let 
	\begin{equation}\label{eq::beta_f}
		\beta^{\epsilon}(f) = 
		\operatorname{argmin}_{\beta} \left\|f - \sum_{\tilde\psi_l \in \mathcal{F}^{*}_{\epsilon}}\beta_l \tilde\psi_l \right\|_{\infty}.
	\end{equation}

	This implies that we can approximate $M$ with a low-rank matrix $M^\epsilon$, with entries given by
	\begin{equation}\label{eq::approx}
		m^{\epsilon}_{ij} \gets \sum_{\tilde\psi_l \in \mathcal{F}^{*}_{\epsilon}}\tilde\psi_l\left(\bm\theta_{i,\cdot}\right) \cdot \beta^{\epsilon}_l(f_j),
	\end{equation}
	such that $|m_{ij} - m^\epsilon_{ij}| \le \epsilon$ for $i=1,\ldots,n$, $j=1,\ldots,p$. Now, let $\Psi$ denote the matrix with $\Psi_{il} = \tilde\psi_l\left(\bm\theta_{i,\cdot}\right)$ and $B$ denote the matrix with entries $B_{lj} = \beta^{\epsilon}_l(f_j)$. Then, the approximation matrix can be compactly written as 
	\[M^{\epsilon}= \Psi B,\]
	with $\Psi \in \mathbb{R}^{n\times J^*(\epsilon)}$ and $B\in\mathbb{R}^{J^*(\epsilon) \times p}$. Thus,  $\operatorname{rank}(M^{\epsilon}) = J^*(\epsilon) \le \min(n,p)$ and
	\[
	\left\|M^{\epsilon} - M\right\|_{\infty} \leq \epsilon.
	\]
	
	Finally, using a variational form of the nuclear norm \citep{srebro2005rank}, we have
	\[
	\frac{1}{\sqrt{np}}\left\|M^{\epsilon}\right\|_{*} = \frac{1}{2}\,\,\min_{UV^{\top} = M^{\epsilon}} \left( \frac{1}{n}\left\|U\right\|_F^2 +  \frac{1}{p}\left\|V\right\|_F^2 \right).
	\]
	From the statement above \eqref{eq::beta_f}, we know that $\|\Psi\|_\infty$ and $\|B\|_\infty$ are both bounded by $C_0$. 
	Thus we have that
	\[
	\frac{1}{\sqrt{np}}\left\|M^{\epsilon}\right\|_{*} \leq \frac{1}{2}\left(\frac{1}{n} \left\|\Psi\right\|_F^2 + \frac{1}{p}\left\|B\right\|_F^2 \right) \leq C_0^2 J^*(\epsilon).
	\]
	Noting that $C_0^2$ is a constant independent of $\epsilon$ gives us our result.
\end{proof}

\section{Deriving the Consistency}\label{proof_upper_bound}

In this section, we shall derive the consistency of our estimator $\widehat{M}$. Recall that $\{(y_t, X_t)\}_{t=1}^N$ are generated by
\begin{equation}\label{eq::supp_data_model}
	y_t = \langle X_t, M\rangle + \xi_t,
\end{equation}
where $\xi_t$ are i.i.d. random errors distributed $N(0,\sigma^2)$, and $M$ is a $n\times p$ matrix. The estimator we consider is defined by
\begin{equation}\label{eq::est_M}
	\begin{aligned}
		\widehat M &\leftarrow  \operatorname{argmin}_{M\in \mathbb R^{n\times p}} \left\{\frac{1}{np}\|M\|_F^2 - \left\langle \frac{2}{N} \sum_{t=1}^{N} y_t X_t, M \right\rangle + \lambda \|M\|_* \right\}\\
		&\equiv \operatorname{argmin}_{M\in \mathbb R^{n\times p}} {L_N(M)}
	\end{aligned}
\end{equation}

We first introduce two technical lemmas, which will play a key role in showing the convergence rate. Proving these lemmas will entail most of the work required for proving this theorem. In Lemma~\ref{lem::upper_bound0}, we derive a deterministic upper bound for the estimation error (under a stochastic condition) as a function of the regularization parameter $\lambda$, when $\lambda$ is sufficiently large (in this Lemma, ``sufficiently large'' is left as a stochastic constraint). In particular, we show that the risk can be decomposed into a misspecification error and a prediction error. Then, in Lemma~\ref{lem::op_bound}, we identify a deterministic value for $\lambda$ such that, with high probability, the condition in Lemma~\ref{lem::upper_bound0} will hold. More specifically we give probabilistic bounds for the operator norm of the stochastic error term in our generative model. We can then combine these to obtain the general oracle inequality in Theorem~\ref{thm::upper_bound}.

Before continuing, we give some additional notation:  For any matrix $Z$, we denote $\|Z\|_{op} = \Lambda_{\max}(Z)$,  where $\Lambda_{\max}^2(Z) = \Lambda_{\max}(Z^TZ)$ is the largest singular value of $Z^TZ$, also known as the operator-norm.


\begin{lemma}\label{lem::upper_bound0}
	Suppose we observe $\{(y_t, X_t)\}_{t=1}^N$ generated by \eqref{eq::supp_data_model}, where $X_t$ are i.i.d. uniformly sampled from $\mathcal{X}$. Further,  assume the underlying true matrix $M \in \mathbb{R}^{n\times p}$ is $\mathcal{F}$-embeddable with Condition \ref{cond::approx} satisfied.   Let $\Delta = N^{-1}\sum_{t=1}^{N}[y_tX_t - \operatorname{E}(y_tX_t)]$. If $\lambda \ge 2\|\Delta\|_{op}$, then 
	\begin{equation}\label{eq::upper_bound0}
		\frac{1}{np}\|\hat{M}-M\|_F^2\le \epsilon^2 + \left(\frac{1+\sqrt{2}}{2}\right)^2 J^*(\epsilon) \lambda^2 np
	\end{equation}
	holds for any $\epsilon>0$. Recall that $J^*(\epsilon)$ is the minimal rank of an approximation matrix $M^\epsilon$ with $\|M^\epsilon-M\|_\infty\ < \epsilon$.
\end{lemma}

\begin{proof}
	
	The proof of this lemma is based on the strong convexity of the loss function $L_N(M)$.

	Consider the the subdifferential of $L_N(M)$, which is the set of matrices of the following form:
	\begin{equation}\label{eq::proof_lem_S1_eq1}
		\partial L_N(M)= \left\{  \frac{2}{np} M - \frac{2}{N}\sum_{t=1}^{N} y_t X_t + \lambda  B, ~~B \in \partial \|M\|_* \right\}.
	\end{equation}
	Thus, the following representation holds for $\widehat{A} \in \partial L_N(\widehat M)$
	\[
	\widehat{A} = \frac{2}{np} \widehat M - \frac{2}{N}\sum_{t=1}^{N} y_t X_t + \lambda \widehat B,
	\]
	for some $\widehat{B}\in \partial \|\widehat M\|_*$. Since $M \mapsto L_N(M)$ is strictly convex, $\widehat{M}$ defined in \eqref{eq::est_M} is the unique minimizer of $L_N(M)$. This implies, $\bm{0} \in \partial L_N(\widehat{M})$. Hence, there exists $\widehat{B}\in \partial \|\widehat M\|_*$ such that $\widehat{A}=\bm{0}$, and thus
	\begin{equation}\label{eq::proof_lem_S1_eq2}
		\langle \widehat{A}, \widehat{M}-M^\epsilon \rangle = \langle \bm{0}, \widehat{M}-M^\epsilon \rangle  = 0.
	\end{equation}
	It further follows that
	\begin{equation}\label{eq::proof_lem_S1_eq3}
		\begin{aligned}
			& \langle \widehat{A}, \widehat M-M^\epsilon \rangle \\
			& = \frac{2}{np}\langle\widehat M, \widehat M-M^\epsilon\rangle - \frac{2}{N}\sum_{t=1}^{N} \langle y_t X_t, \widehat M- M^\epsilon\rangle + \lambda \langle \widehat B, \widehat M - M^\epsilon\rangle = 0.
		\end{aligned}
	\end{equation}

	$M^\epsilon\in \mathbb{R}^{n\times p}$ is the approximation matrix with $\operatorname{rank}(M^\epsilon) = J^*(\epsilon)$. So, it has spectral representation $M^\epsilon = \sum_{j=1}^{J^*(\epsilon)}\sigma_ j u_j v_j^T$ where $u_j\in\mathbb{R}^n$  and $v_j \in\mathbb{R}^{p}$, $j=1,...,J^*(\epsilon)$, are orthonormal vectors, and $\sigma_j$ are the singular values of $M^\epsilon$.  Let $U$ and $V$ denote the linear span of $\{u_1,...,u_{J^*(\epsilon)}\}$ and $\{v_1,...,v_{J^*(\epsilon)}\}$ respectively. Then, the subdifferential of $\|M^\epsilon\|_*$ can be represented by the following set of matrices \citep{watson1992characterization}: 
	\[
	\partial\|M^\epsilon\|_* = \left\{\sum_{j=1}^{J^*(\epsilon)} u_jv_j^T + P_{U^\bot}WP_{V^\bot}: \|W\|_{op} \le 1 \right\},
	\]
	where $U^\bot$ denotes the orthogonal complements of $U$ and $P_{U^\bot}$ denotes  the projection on the linear vector subspace $U^\bot$. The same argument applies to $V$ and $P_{V^\bot}$. Thus, $B^\epsilon \in \partial \|M^\epsilon\|_*$ can be represented as 
	\begin{equation}\label{eq::proof_lem_S1_eq4}
		B^\epsilon = \sum_{j=1}^{J^*(\epsilon)} u_jv_j^T + P_{U^\bot}WP_{V^\bot}
	\end{equation}
	for arbitrary matrix $W$ having $\|W\|_{op} \le 1$. Due to the trace duality,
	there exists $W$ with $\|W\|_{op}\le 1$ such that 
	\begin{equation}\label{eq::proof_lem_S1_eq5}
		\langle  P_{U^\bot}WP_{V^\bot}, \widehat M - M^\epsilon\rangle = \langle  P_{U^\bot}WP_{V^\bot}, \widehat M \rangle = \langle  W, P_{U^\bot}\widehat{M}P_{V^\bot} \rangle = \|P_{U^\bot}\widehat M P_{V^\bot}\|_*
	\end{equation}

	So, it follows from  \eqref{eq::proof_lem_S1_eq3} that
	\begin{equation}\label{eq::proof_lem_S1_eq6}
		\begin{aligned}
			& \frac{2}{np}\langle \widehat M-M, \widehat M - M^\epsilon \rangle + \frac{2}{np}\langle M, \widehat M - M^\epsilon \rangle + \lambda \langle \widehat B- B^\epsilon, \widehat M - M^\epsilon\rangle  \\
			&  = \frac{2}{N}\sum_{t=1}^{N} \langle \operatorname{E}( y_tX_t), \widehat M- M^\epsilon\rangle - \lambda \langle B^\epsilon, \widehat M - M^\epsilon\rangle + \frac{2}{N}\sum_{t=1}^{N} \langle y_t X_t -\operatorname{E}(y_t X_t), \widehat M- M^\epsilon\rangle
		\end{aligned}
	\end{equation}

	Due to the monotonicity of subdifferentials of convex functions $M \mapsto \|M\|_*$,  $\langle \widehat B- B^\epsilon, \widehat M - M^\epsilon\rangle  \ge 0$. So, \eqref{eq::proof_lem_S1_eq6} can be further simplified:
	\begin{equation}\label{eq::proof_lem_S1_eq7}
		\begin{aligned}
			\frac{2}{np}\langle \widehat M-M, \widehat M - M^\epsilon \rangle 
			&  \le  - \lambda \langle B^\epsilon, \widehat M - M^\epsilon\rangle +  2\langle \Delta, \widehat M- M^\epsilon\rangle \\
			\underrightarrow{\eqref{eq::proof_lem_S1_eq4}}
			& = - \lambda \langle  \sum_{j=1}^{J^*(\epsilon)} u_jv_j^T + P_{U^\bot}WP_{V^\bot}, \widehat M - M^\epsilon\rangle +  2\langle \Delta, \widehat M- M^\epsilon\rangle\\
			\underrightarrow{\eqref{eq::proof_lem_S1_eq5}}
			& = - \lambda \langle  \sum_{j=1}^{J^*(\epsilon)} u_jv_j^T, \widehat M - M^\epsilon\rangle +  2\langle \Delta, \widehat M- M^\epsilon\rangle -\lambda \|P_{U^\bot}\widehat M P_{V^\bot}\|_*\\
		\end{aligned}
	\end{equation}
	where $\Delta = N^{-1}\sum_{t=1}^{N}[y_tX_t - \operatorname{E}(y_tX_t)]$.

	By arithmetic, we see that the left-hand side of \eqref{eq::proof_lem_S1_eq7} is equal to:
	\begin{equation}\label{eq::proof_lem_S1_eq8}
		\begin{aligned}
			2\langle \widehat M-M, \widehat M - M^\epsilon \rangle  
			& = \langle \widehat{M}-M, \widehat{M}-M + M - M^\epsilon  \rangle   + \langle\widehat{M}-M^\epsilon + M^\epsilon -M, \widehat{M}-M^\epsilon \rangle\\
			& = \|\widehat M - M\|_F^2 - \|M^\epsilon - M\|_F^2 + \|\widehat M - M^\epsilon\|_F^2.
		\end{aligned}
	\end{equation}

	As for the right side of \eqref{eq::proof_lem_S1_eq7}, we use the following facts:
	\begin{equation}\label{eq::proof_lem_S1_eq9}
		\|\sum_{j=1}^{J^*(\epsilon)} u_jv_j^T\|_{op} = 1 \quad \text{and}\quad \langle \sum_{j=1}^{J^*(\epsilon)} u_jv_j^T ,  \widehat M - M^\epsilon\rangle = \langle \sum_{j=1}^{J^*(\epsilon)} u_jv_j^T,  P_U(\widehat M - M^\epsilon)P_V\rangle.
	\end{equation}

	Given \eqref{eq::proof_lem_S1_eq8}-\eqref{eq::proof_lem_S1_eq9}, \eqref{eq::proof_lem_S1_eq7} becomes
	\begin{equation}\label{eq::proof_lem_S1_eq10}
		\begin{aligned}
			&\frac{1}{np}\|\widehat M - M\|_F^2  + \frac{1}{np}\|\widehat M - M^\epsilon\|_F^2  + \lambda \|P_{U^\bot}\widehat M P_{V^\bot}\|_*\\
			&\le - \lambda \langle  \sum_{j=1}^{J^*(\epsilon)} u_jv_j^T, \widehat M - M^\epsilon\rangle + \frac{1}{np}\|M^\epsilon - M\|_F^2  +  2\langle \Delta, \widehat M- M^\epsilon\rangle \\
			& \le \lambda \|P_U(M^\epsilon-\widehat M)P_V\|_*+ \frac{1}{np}\|M^\epsilon - M\|_F^2  +  2\langle \Delta, \widehat M- M^\epsilon\rangle , \\
		\end{aligned}
	\end{equation}
	where the last inequality is due to $|\langle M_1, M_2 \rangle| \le \|M_1\|_{op} \times \|M_2\|_*$.
	
	In \eqref{eq::proof_lem_S1_eq10}, the stochastic error term $\langle \Delta, \widehat M - M^\epsilon \rangle$ can be decomposed:
	\begin{equation}\label{eq::proof_lem_S1_eq11}
		\begin{aligned}
			\langle \Delta, \widehat M - M^\epsilon\rangle & = \langle \mathcal P_{M^\epsilon}(\Delta), \widehat M-M^\epsilon\rangle + \langle P_{U^\bot}\Delta P_{V^\bot}, \widehat M-M^\epsilon\rangle \\
			& = \langle \mathcal P_{M^\epsilon}(\Delta), \mathcal P_{M^\epsilon}(\widehat M - M^\epsilon)\rangle  + \langle \mathcal P_{M^\epsilon}(\Delta), \mathcal P_{U^\bot}(\widehat M - M^\epsilon)P_{V^\bot}\rangle \\
			& ~~~ + \langle P_{U^\bot}\Delta P_{V^\bot}, \mathcal P_{M^\epsilon}(\widehat M)\rangle + \langle P_{U^\bot}\Delta P_{V^\bot}, P_{U^\bot}\widehat M P_{V^\bot}\rangle -  \langle P_{U^\bot}\Delta P_{V^\bot}, M^\epsilon\rangle\\
			& = \langle \mathcal P_{M^\epsilon}(\Delta), \mathcal P_{M^\epsilon}(\widehat M - M^\epsilon)\rangle  + \langle P_{U^\bot}\Delta P_{V^\bot}, P_{U^\bot}\widehat M P_{V^\bot}\rangle 
		\end{aligned}
	\end{equation}
	where $\mathcal P_{M^\epsilon}(\Delta) = \Delta - P_{U^\bot}\Delta P_{V^\bot}$. So it can be upper bounded by:
	\begin{equation}\label{eq::proof_lem_S1_eq12}
		\begin{aligned}
			|\langle \Delta, \widehat M - M^\epsilon\rangle |
			& \le \|\mathcal P_{M^\epsilon}(\Delta)\|_F \|\mathcal P_{M^\epsilon}(\widehat M - M^\epsilon)\|_F +  \|P_{U^\bot}\Delta P_{V^\bot}\|_{op}\|P_{U^\bot}\widehat M P_{V^\bot}\|_*\\
			& \le \|\mathcal P_{M^\epsilon}(\Delta)\|_F \|\widehat M - M^\epsilon\|_F +  \|P_{U^\bot}\Delta P_{V^\bot}\|_{op}\|P_{U^\bot}\widehat M P_{V^\bot}\|_*\\
			& \le \sqrt{2J^*(\epsilon)} \|\Delta\|_{op} \|\widehat M - M^\epsilon\|_F +  \|\Delta\|_{op}\|P_{U^\bot}\widehat M P_{V^\bot}\|_* .
		\end{aligned}
	\end{equation}

	The last inequality is due to the facts that
	\[
	\begin{aligned}
		\|\mathcal P_{M^\epsilon}(\Delta)\|_F
		&\le \sqrt{\operatorname{rank}(\mathcal P_{M^\epsilon}(\Delta))} \|\Delta\|_{op}  = \sqrt{\operatorname{rank}(P_{U^\bot}\Delta P_{V} + P_U\Delta)} \|\Delta\|_{op} \\
		& \le \sqrt{2\operatorname{rank}(M^\epsilon)} \|\Delta\|_{op} = \sqrt{2J^*(\epsilon)} \|\Delta\|_{op}
	\end{aligned}
	\]
	and $\|P_{U^\bot}\Delta P_{V^\bot}\|_{op} \le \|\Delta\|_{op}$.
	
	Meanwhile, the first term in the right-hand side of \eqref{eq::proof_lem_S1_eq10} can also be bounded:
	\begin{equation} \label{eq::proof_lem_S1_eq13}
		\|P_U(M^\epsilon-\widehat M)P_V\|_*  \le \sqrt{\operatorname{rank}(M^\epsilon)}\|P_U(M^\epsilon-\widehat M)P_V\|_F \le \sqrt{J^*(\epsilon)}\|M^\epsilon-\widehat M\|_F.
	\end{equation}

	Combining \eqref{eq::proof_lem_S1_eq12} - \eqref{eq::proof_lem_S1_eq13},  \eqref{eq::proof_lem_S1_eq10} becomes
	\begin{equation}
		\begin{aligned}
			\frac{1}{np}\|\widehat M - M\|_F^2  + &\frac{1}{np}\|\widehat M - M^\epsilon\|_F^2  + (\lambda -2\|\Delta\|_{op})  \|P_{U^\bot}\widehat M P_{V^\bot}\|_*\\
			& \le \lambda \sqrt{J^*(\epsilon)}\|M^\epsilon-\widehat M\|_F + \epsilon^2 + 2\sqrt{2J^*(\epsilon)} \|\Delta\|_{op}\|\widehat M - M^\epsilon\|_F.
		\end{aligned}
	\end{equation}

	If $\lambda \ge 2\|\Delta\|_{op}$, then 
	\begin{equation}
		\frac{1}{np}\|\widehat M - M\|_F^2  + \frac{1}{np}\|\widehat M - M^\epsilon\|_F^2  \le \epsilon^2 + (1+\sqrt{2}) \lambda \sqrt{J^*(\epsilon)}\|\widehat M - M^\epsilon\|_F
	\end{equation}
	which implies
	\begin{equation}
		\begin{aligned}
			\frac{1}{np}\|\widehat M - M\|_F^2 
			& \le \epsilon^2 + (1+\sqrt{2}) \lambda \sqrt{J^*(\epsilon)}\|\widehat M - M^\epsilon\|_F - \frac{1}{np}\|\widehat M - M^\epsilon\|_F^2\\
			& \le \epsilon^2 + \left(\frac{1+\sqrt{2}}{2}\right)^2 J^*(\epsilon) \lambda^2 np
		\end{aligned}
	\end{equation}
	as claimed.
\end{proof}


The result in Lemma~\ref{lem::upper_bound0} still contains regularization parameter $\lambda$. When $\lambda$ is selected too large, then entries of $\widehat{M}$ will be overly shrunk toward zero and give poor reconstruction error. If $\lambda$ is too small, then our constraint, $\lambda \ge 2\|\Delta\|_{op}$, will not be satisfied. Thus, it is important to identify a minimal value for $\lambda$ such that $\lambda \ge 2\|\Delta\|_{op}$ with high probability. Here, we introduce the second lemma, which gives an upper bound for $\|\Delta\|_{op}$.

\begin{lemma}\label{lem::op_bound}
	Consider the same data generating mechanism as in Lemma~\ref{lem::upper_bound0}, with $X_t$ are i.i.d. uniformly sampled from $\mathcal{X}$. Then, there exists constant $c_1$ (dependent on $\sigma$ and $\|M\|_\infty$) such that 
	\begin{equation}\label{eq::bound_op}
		\|\Delta\|_{op}
		\le c_1 \left[\sqrt{\frac{\log(n+p)}{N(n\wedge p)}}+ \sqrt{\log\left(\frac{8 (n\wedge p)}{3\sigma^2}\right)}\frac{\log(n+p)}{N}\right]
	\end{equation}
	with probability at least $1-2(n+p)^{-1}$. 
	
	Furthermore, when $N \ge (n\wedge p)\log^2(n+p)$, we have $\|\Delta\|_{op} \le 2c_1\sqrt{\frac{\log(n+p)}{N(n\wedge p)}}$ with probability at least $1-2(n+p)^{-1}$. 
\end{lemma}

To derive the bound of the stochastic error $\Delta$, we shall use the matrix version of Bernstein's inequality. We now use 2 propositions from \citet{van2016estimation}. For completeness, we include statements of the propositions here below. 

\begin{proposition}\label{prop::sara_thm_9.7}
	Let $\{Z_t\}_{t=1}^N$ be i.i.d. $n \times p$ matrices that satisfy for some $\alpha\ge 1$ and all $t$
	\[
	\operatorname{E}Z_t = \mathbf{0},  \quad  K: = \| \|Z_t\|_{op} \|_{\Psi(\alpha)} < \infty, 
	\]
	where $\|\cdot\|_{\Psi(\alpha)}$ is the $\Psi(\alpha)$-Orlicz norm defined as $\|z\|_{\Psi(\alpha)}:=\inf\left\{c>0: \operatorname{E}\exp\left(\frac{|z^\alpha|}{c^{\alpha}} \right) \le 2 \right\}$ for a random variable $z \in \mathbb{R}$. Define
	\[
	R^2 := \max \left\{ \left\|  \frac{1}{N}\sum_{t=1}^{N} \operatorname{E}Z_tZ_t^T \right\|_{op},  \left\|  \frac{1}{N}\sum_{t=1}^{N} \operatorname{E}Z_t^TZ_t \right\|_{op} \right\}.
	\]
	Then for a constant $\tilde{c}$ and for all $h > 0$,
	\[
	\mathbb P \left(\left\|\frac{1}{N} \sum_{t=1}^{N} Z_t \right\|_{op}  \ge \tilde{c} R \sqrt{\frac{h+\log(n+p)}{N}} +  \tilde{c}\log^{1/\alpha}\left(\frac{K }{R}\right)\left(\frac{h+\log(n+p)}{N}\right)\right) \le \exp(-h).
	\]
\end{proposition}

\begin{proposition}\label{prop::sara_thm_9.5}
	Let $\{Z_t\}_{t=1}^N$ be $n \times p$ matrices that satisfy for a constant $K_1$ 
	\[
	\operatorname{E}Z_t = \mathbf{0},  \quad   \max_{1\le t \le N}\|Z_t\|_{op} \le  K_1.
	\]
	With the same definition for $R$ as in Proposition~\ref{prop::sara_thm_9.7}
	Then for all $h > 0$,
	\[
	\mathbb P \left(\left\|\frac{1}{N} \sum_{t=1}^{N} Z_t \right\|_{op}  \ge \sqrt{2}R \sqrt{\frac{h+\log(n+p)}{N}} + \frac{K_1[h+\log(n+p)]}{3N}\right) \le \exp(-h).
	\]
\end{proposition}

Given the above results, we now prove Lemma~\ref{lem::op_bound}.

\begin{proof}[Proof of Lemma~\ref{lem::op_bound}]
	$\|\Delta\|_{op}$ can be decomposed into two parts as below and we shall bound each part respectively.
	\begin{equation}\label{eq::proof_lem_S2_eq1}
		\begin{aligned}
			\|\Delta\|_{op} 
			& = \left\|\frac{1}{N} \sum_{t=1}^{N} [y_tX_t - \operatorname{E}(y_tX_t)]\right\|_{op} \\
			& = \left\|\frac{1}{N} \sum_{t=1}^{N} [\xi_tX_t -\operatorname{E}(\xi_tX_t)+ \operatorname{tr}[M^TX_t]X_t - \operatorname{E}(\operatorname{tr}(M^TX_t)X_t ) \right\|_{op} \\
			& \le \left\|\frac{1}{N} \sum_{t=1}^{N} \xi_t X_t\right\|_{op} + \left\|\frac{1}{N} \sum_{t=1}^{N} \left(\operatorname{tr}(M^TX_t)X_t - \operatorname{E}(\operatorname{tr}(M^TX_t)X_t)\right)\right\|_{op}\\ 
			& = I_1 + I_2.
		\end{aligned}
	\end{equation}
	
	We use Proposition~\ref{prop::sara_thm_9.7} to bound $I_1$. Let $Z_{1,t} = \xi_tX_t$. Since $\xi_t \sim_{i.i.d.} N(0,\sigma^2)$ and $X_t$ are i.i.d. uniformly sampled from $\mathcal{X}$ with $\xi_t \indep X_t$, $\{Z_{1,t}\}_{t=1}^N$ are i.i.d. $n\times p$ matrices having
	\[
	\operatorname{E}Z_{1,t} =  \mathbf{0}, \quad  K:= \| \| Z_{1,t} \|_{op} \|_{\Psi(\alpha)}  =\| \xi_t \|_{\Psi(\alpha)}.
	\]
	
	For a normal variable $z \sim N(0,1)$, we have $\operatorname{E}\exp(z^2/c^2) = c/\sqrt{c^2-2}$ when $c>\sqrt{2}$. Thus, $\operatorname{E}\exp(z^2/c^2) \le 2 \Rightarrow c \ge \sqrt{8/3}$. So, $K= \| \xi_t \|_{\Psi(2)} = \sqrt{8/3}$. Let 
	\[
	\begin{aligned}
		R^2 & := \max\left\{ \left\| \frac{1}{N}\sum_{t=1}^{N}\operatorname{E}Z_{1,t}Z_{1,t}^T\right\|_{op},  \left\| \frac{1}{N}\sum_{t=1}^{N}\operatorname{E}Z_{1,t}^TZ_{1,t}\right\|_{op}  \right\} \\
		& = \sigma^2 \max\left\{ \left\| \frac{1}{N}\sum_{t=1}^{N}\operatorname{E}(X_tX_t^T)\right\|_{op},  \left\| \frac{1}{N}\sum_{t=1}^{N}\operatorname{E}(X_t^TX_t)\right\|_{op}  \right\} \\
		& = \frac{\sigma^2}{n \wedge p}.
	\end{aligned}
	\]
	
	Due to Proposition~\ref{prop::sara_thm_9.7}, for some $\tilde{c}$ and for all $h>0$, we have
	\begin{equation}\label{eq::bound_I1}
		\mathbb P \left(I_1\ge \tilde{c} \sigma \sqrt{\frac{h+\log(n+p)}{N(n\wedge p)}} +  \tilde{c}\sqrt{\frac{1}{2}\log\left(\frac{8 (n\wedge p) }{3\sigma^2}\right)}\left(\frac{h+\log(n+p)}{N}\right)\right) \le \exp(-h).
	\end{equation}

	Similarly, we use Proposition~\ref{prop::sara_thm_9.5} to bound $I_2$. Let $Z_{2,t} = \operatorname{tr}(M^TX_t)X_t - \operatorname{E}\left(\operatorname{tr}(M^TX_t)X_t\right)$, where $\operatorname{E}\left(\operatorname{tr}(M^TX_t)X_t\right) = \frac{1}{np}M$. So, $\operatorname{E}(Z_{2,t})=\mathbf{0}$, and
	\[
	\|Z_{2,t}\|_{op} \le \left\| \operatorname{tr}(M^TX_t)X_t \right\|_{op}  + \left\|\operatorname{E}(\operatorname{tr}(M^TX_t)X_t) \right\|_{op} \le 2\|M\|_\infty.
	\] 
	Let $K_1 = 2\|M\|_\infty$. Then $\max_{1 \le t\le N} \|Z_{2,t}\|_{op} \le K_1$. Consider,
	\[
	\operatorname{E}(Z_{2, t}Z_{2, t}^T) = \operatorname{E}\left[\operatorname{tr}(M^TX)^2XX^T\right] - \left(\frac{1}{np}\right)^2MM^T,
	\]
	\[
	\operatorname{E}(Z_{2, t}^TZ_{2, t}) = \operatorname{E}\left[\operatorname{tr}(M^TX)^2X^TX\right] - \left(\frac{1}{np}\right)^2M^TM.
	\]
	
	Then,
	\[
	\begin{aligned}
		\left\|  \operatorname{E}(Z_{2, t}Z_{2, t}^T) \right\|_{op} & \le \left\| \operatorname{E}\left[\operatorname{tr}(M^TX)^2XX^T\right] \right\|_{op}  + \left\| \left(\frac{1}{np}\right)^2MM^T \right\|_{op} \\
		& \le \|M\|_\infty^2/n + \frac{\|M\|_\infty^2}{np} \le 2 \|M\|_\infty^2/n,
	\end{aligned}
	\]
	and similarly $\left\|  \operatorname{E}(Z_{2, t}^TZ_{2, t}) \right\|_{op} \le 2 \|M\|_\infty^2/p$. Let
	\[
	\begin{aligned}
		R_1^2  := \max\left\{ \left\| \frac{1}{N}\sum_{t=1}^{N}\operatorname{E}(Z_{2, t}Z_{2, t}^T)\right\|_{op},  \left\| \frac{1}{N}\sum_{t=1}^{N}\operatorname{E}(Z_{2, t}^TZ_{2, t})\right\|_{op}  \right\}  \le \frac{2\|M\|_\infty^2}{n \wedge p}.
	\end{aligned}
	\]

	Then, applying Proposition~\ref{prop::sara_thm_9.5}, we have
	\begin{equation}\label{eq::bound_I2}
		\mathbb P \left(I_2  \ge 2\|M\|_\infty \sqrt{\frac{h+\log(n+p)}{N(n \wedge p)}} + \frac{2\|M\|_\infty[h+\log(n+p)]}{3N}\right) \le \exp(-h).
	\end{equation}

	Combining the results of \eqref{eq::bound_I1} and \eqref{eq::bound_I2}, for all $h>0$
	\begin{equation}\label{eq::proof_lem_S2_eq10}
		\begin{aligned}
			\mathbb P & \left( \|\Delta\|_{op}  \ge (\tilde{c}\sigma  + 2\|M\|_\infty) \left[ \sqrt{\frac{h+\log(n+p)}{N(n\wedge p)}}  + \sqrt{\frac{1}{2}\log\left(\frac{8 (n\wedge p)}{3\sigma^2}\right)} \left(\frac{h+\log(n+p)}{N}\right) \right] \right) \\
			& \le 2\exp(-h).
		\end{aligned}
	\end{equation}
	
	Select $h = \log (n+p)$ and let $c_1=\sqrt{2}(\tilde{c}\sigma  + 2\|M\|_\infty) $, then
	\begin{equation}\label{eq::proof_lem_S2_eq11}
		\mathbb P\left(\|\Delta\|_{op}   \ge c_1  \left[\sqrt{\frac{\log(n+p)}{N(n\wedge p)}}+ \sqrt{\log\left(\frac{8 (n\wedge p)}{3\sigma^2}\right)}\frac{\log(n+p)}{N}\right]\right)\le 2(n+p)^{-1}.
	\end{equation}

	In particular, if $N \ge (n\wedge p)\log^2(n+p)$, we have 
	\[
	\mathbb P\left(\|\Delta\|_{op}  \ge 2c_1\sqrt{\frac{\log(n+p)}{N(n\wedge p)}}\right) \le 2(n+p)^{-1},
	\]
	as desired.
	
\end{proof}


Based on Lemma~\ref{lem::upper_bound0} and \ref{lem::op_bound}, it is straightforward to prove Theorem~\ref{thm::upper_bound}.

\begin{proof}[Proof of Theorem~\ref{thm::upper_bound}]
	When $N \ge (n\wedge p)\log^2(n+p)$,  we choose $\lambda$ of the following form
	\begin{equation}\label{eq::lambda}
		\lambda = C_2 \sqrt{\frac{\log(n+p)}{N(n\wedge p)}}
	\end{equation}
	where  $C_2>0$ is a constant with $C_2 \ge 4c_1$,  where $c_1$ is defined in Lemma~\ref{lem::op_bound} that only depends on $\sigma$ and $\|M\|_\infty$.  Following from \eqref{eq::upper_bound0}, then
	\begin{equation}\label{eq::upper_bound1}
		\begin{aligned}
			\frac{1}{np}\|\widehat M-M\|_F^2 \le C_2^2\left(\frac{1+\sqrt{2}}{2}\right)^2\frac{(n\vee p)\log(n+p)}{N}J^*(\epsilon) + \epsilon^2.
		\end{aligned}
	\end{equation}
	holds with probability $1-2(n+p)^{-1}$. 
	
	This completes the proof.
\end{proof}

\section{Proof of Lemma~\ref{lem::J_star_bound}}\label{proof_J_star_bound}

We begin with an outline of the proof. To form our set of basis functions, we will tessellate our domain $\mathbb{X}^K$ with $\infty$-norm balls, and use a Taylor series centered at an arbitrary point within each ball to get a uniform approximation for functions in that ball. For a fixed center point, the Taylor series is a linear combination of fixed basis functions. To obtain our full set of basis functions, we will collect all of the terms in all of those Taylor series. We now formalize this:

\begin{proof}
	For functions satisfying Condition~\eqref{cond::bound_derivative}, we consider a Taylor series approximation to $f \in \mathcal{F}(L,\gamma,K)$ of order $L$ at a point $\mathbf{x^0} \in \mathbb{R}_{[0,1]}^K$, that is \[
	T_{\mathbf{x}^0}f(\mathbf{x}) = f(\mathbf{x^0}) + \sum_{l\le L-1} \frac{1}{l!} \nabla^lf(\mathbf{x^0})(\mathbf{x}-\mathbf{x^0})^l,
	\] 
	where $l!=l_1! \ldots l_k!$, $\nabla^l f(\mathbf{x}) = \frac{\partial^l f}{\partial x_1^{l_1}\cdots \partial x_k^{l_k} }$ and $\mathbf{x}^l = x_1^{l_1}\cdots x_k^{l_k}$ over all combinations with $l_1+ \cdots +l_k=l$. There exists $\mathbf{x'} = (x_1',...,x_K')^T \in \mathbb{R}_{[0,1]}^K$ in a neighborhood of radius $\|\mathbf{x} -\mathbf{x}^0\|_2$ centered at $\mathbf{x}^0$ such that the approximation error obeys
	\begin{equation}\label{eq::taylor_approx}
		\begin{aligned}
			\left|f(\mathbf{x}) - T_{\mathbf{x}^0}f(\mathbf{x})\right| &\leq \left|\sum_{L_1+ \cdots +L_K=L}\frac{1}{L_1!\ldots L_k!}
			\times \frac{\partial^L f(\mathbf{x'})}{\partial {x_1'}^{L_1}\cdots \partial {x_K'}^{L_K} } |x_1-x_1^0|^{L_1}\cdots|x_K-x_K^0|^{L_K}\right|\\
			\underrightarrow{\text{Condition~\ref{cond::bound_derivative}}}
			& \le \gamma \left|\sum_{L_1+ \cdots +L_K=L}\frac{|x_1-x_1^0|^{L_1}\cdots|x_K-x_K^0|^{L_K} }{L_1!\ldots L_k!} \right| \\
			\underrightarrow{\text{Multinomial Theorem}} 
			& = \frac{\gamma}{L!} \left( |x_1-x_1^0|+ ... + |x_K-x_K^0|\right)^L
		\end{aligned}
	\end{equation}

	If we consider the approximation error within an $\infty$-norm ball of radius $d$ (and choose any point in that ball as $\mathbf{x}^0$), then $|x_k-x_k^0|\le d$ for $k=1,...,K$.  \eqref{eq::taylor_approx} has
	\begin{equation}\label{eq::taylor_approx_error}
		\left|f(\mathbf{x}) - T_{\mathbf{x}^0}f(\mathbf{x})\right| \le  \frac{\gamma}{L!} K^Ld^L.
	\end{equation}

	Thus, to get an approximation error of $\epsilon$, let $\frac{\gamma}{L!} K^Ld^L = \epsilon$, we need to divide the space into balls of radius
	\begin{equation}\label{eq::radius}
		d = \sqrt[L]{\frac{L!}{\gamma K^L}} \times \epsilon^{1/L}.
	\end{equation}

	As the support $\mathbb{R}_{[0,1]}^K$ is bounded by 1, we need $(1/d)^K$ balls with radius $d$ (in $\infty$-norm) to cover the entirety of $\mathbb{X}^K$, resulting in  $\binom{K+L}{L}(1/d)^K $ total terms to get an approximation error $\epsilon$ (above Taylor series approximation contains $\binom{K+L}{L}$ terms). If we select balls of radius $d$ in \eqref{eq::radius}, this gives us a total number of terms in our linear expansion
	\[
	J^*(\epsilon) = \binom{K+L}{L}\left(\frac{L!}{\gamma K^L}\right)^{-K/L} \epsilon^{-K/L}.
	\]
	That is, $J^*(\epsilon) = O\left(\epsilon^{-K/L} \right)$.
\end{proof}

\section{Proof of Theorem~\ref{thm::upper_bound2}}\label{proof_upper_bound2}

The proof of this theorem is quite straightforward by connecting a few pieces we have already built.

\begin{proof}
	Given Condition~\ref{cond::bound_derivative} and Lemma~\ref{lem::J_star_bound}, we have $J^*(\epsilon) = C_3 \epsilon^{-\frac{K}{L}}$ for some constant $C_3$ relying on $\gamma,K$, and $L$. Plugging in this to the upper bound in Theorem~\ref{thm::upper_bound}, the upper bound~\eqref{eq::upper_bound1} then becomes
	\begin{equation}\label{eq::upper_bound2}
		C_2^2C_3\left(\frac{1+\sqrt{2}}{2}\right)^2\frac{(n\vee p)\log(n+p)}{N} \epsilon^{-\frac{K}{L}} + \epsilon^2,
	\end{equation}
	which is optimized at
	\begin{equation}\label{eq::optimize_epsilon}
		\begin{aligned}
			&\frac{(n\vee p)\log(n+p)}{N} \epsilon^{-\frac{K}{L}} = \epsilon^2\\
			\Rightarrow ~~ & \epsilon = \left(\frac{(n\vee p)\log(n+p)}{N}\right)^{\frac{L}{2L+K}}.
		\end{aligned}
	\end{equation}
	So, we have
	\begin{equation}\label{eq::upper_bound3}
		\frac{1}{np}\|\widehat{M} - M\|_F^2 \le C^* \left(\frac{(n\vee p)\log(n+p)}{N}\right)^{\frac{2L}{2L+K}}
	\end{equation}
	with probability at least $1-2(n+p)^{-1}$ with $C^* = C_2^2C_3 \left(\frac{1+\sqrt{2}}{2}\right)^2 + 1$. Equivalently, we can say
	\[
	\frac{1}{np}\|\widehat M - M\|^2_F =O_P\left(\left[\frac{(n\vee p)\log(n+p)}{N}\right]^{\frac{2L}{2L+K}}\right),
	\]
	as claimed.
\end{proof}

\section{Deriving the Minimax Lower Bound}\label{proof_lower_bound}

In this section, we derive the minimax lower bound for estimation within $M(L,\gamma,K)$: We show that the convergence rate in Theorem~\ref{thm::upper_bound2} is optimal (up to log terms). 

Recall that we assume the true $M$ belongs to the following class of matrices:
\begin{equation}\label{eq::supp_mat_class}
	\mathcal M(L,\gamma,K) := \{M \in \mathbb{R}^{n\times p}: m_{ij} = f_j(\bm{\theta}_{i,\cdot}), ~ \bm{\theta}_{i,\cdot} \in \mathbb{R}^{K}_{[0,1]} , f_j \in \mathcal{F}(L,\gamma, K), \forall j \le p\},
\end{equation}
where $\mathcal F(L,\gamma,K)$ is a class of functions with bounded derivatives:
\begin{equation}\label{eq::func_class}
	\mathcal F(L,\gamma,K) := \left\{f: \left|\frac{\partial^L}{\partial x_1^{L_1}\cdots x_K^{L_K}} f(\mathbf{x}) \bigg\vert_{\mathbf{x}=\mathbf{x^0}}\right| \leq \gamma, ~ \forall \mathbf{x}^0 \in \mathbb{R}^{K}_{[0,1]},~ \sum_{k=1}^{K}L_k=L\right\}.
\end{equation}

For simplicity of notation, let $\bm{\theta}_{i} := \bm{\theta}_{i,\cdot} \in \mathbb{R}^K_{[0,1]}$ denote the $i$-th row vector of the embeddings $\Theta \in \mathbb{R}^{n \times K}$ in this section.

We shall obtain the lower bound based on information theory. The bound is with respect to $\|\cdot\|_F^2$-risk. We pose things in terms of the error in a multi-way hypothesis testing problem, where the set of testing hypotheses should be a suitably large packing set for $\mathcal{M}(L,\gamma,K)$. In this section, we first show the existence of such a suitably large packing set. Then, we apply Yang's method \citep{yang1999information} to prove the main results in Theorem~\ref{thm::lower_bound}.

\subsection{Constructing the $2\delta_{N,n,p}$-packing Set}

For $M \in \mathcal{M}(L,\gamma, K)$, the risk of the estimator can be written as

\[
\frac{1}{np}\|\widehat{M} - M\|_F^2 = \frac{1}{np} \sum_{i=1}^{n} \sum_{j=1}^{p}\left[\hat{m}_{ij} - f_j(\bm{\theta}_i)  \right]^2.
\]

This is to say, bounding $\frac{1}{np}\|\widehat{M} - M\|_F^2$ can be viewed as a classical nonparametric regression problem. So, we follow the construction of many hypotheses as in Section~2.6 of \cite{tsybakov2009introduction}. However, here we are working in a multi-dimensional setting, i.e., $\bm{\theta}_i \in \mathbb{R}^{K}$, $K \ge 1$.

In giving our packing set, we will work with combinations of ``bump functions''. To define these, we need an archetypal ingredient --- the bump functions that we will use:
\begin{equation}\label{eq::univar_func}
	\varphi(u) = c_L e \times \exp \left ( -\frac{1}{1-4u^2} \right ), \quad u \in  (-1/2, 1/2),
\end{equation}
which is infinitely differentiable and vanishes outside of $(-1/2,1/2)$; $c_L>0$ is a tiny constant that only depends on $L$ such that $|\partial^l \varphi(u) / \partial u^l | \le 1$, $\forall l=0,1,...,L$. Meanwhile, since $ \int_{-1/2}^{1/2}e^2 \exp^2 \left ( -\frac{1}{1-4u^2} \right ) du >0.49$, (it is actually very close to 0.5), we have $\|\varphi\|_2^2 := \int_{-1/2}^{1/2}\varphi^2(u)du > 0.49 c_L^2$. In addition, the maximum value of this function is $\sup_u|\varphi(u)| = \varphi (0) = c_L$.

Now, we shall work under the multidimensional setting. We use bold letters to refer to multivariate indices and regular letters to refer to the indices of each coordinate. Let $\bm i = (i_1,...,i_K) \in \{1,2,...,\sqrt[K]{n}\}^K$ having $\sum_{\bm i=(i_1,...,i_K)} 1 = \sum_{i_1=1}^{\sqrt[K]{n}}...\sum_{i_K=1}^{\sqrt[K]{n}}1=n$, where $\sqrt[K]{n}$ is assumed to be an integer. Suppose that the observed embeddings follows a fixed equispaced design, i.e., $\bm\theta_{\bm i} = \bm\theta_{(i_1,...,i_K)}= (\theta_{i_1},...,\theta_{i_K})^T = (\frac{i_1}{\sqrt[K]{n}}, ..., \frac{i_K}{\sqrt[K]{n}})^T$. Consider a multivariate function $\Phi_{\bm d}:\mathbb{R}^K\to\mathbb{R}$,
\begin{equation}\label{eq::multivar_func}
	\begin{aligned}
		\Phi_{\bm d}(\bm \theta_{\bm i}) 
		& =  \gamma b^{-L/K} \prod_{k=1}^{K}\varphi_{d_k}(\theta_{i_k}) \\
		& := \gamma b^{-L/K} \prod_{k=1}^{K}\varphi(\sqrt[K]{b}\theta_{i_k}-d_k+1/2),
	\end{aligned}
\end{equation}
where $\bm d = (d_1,...,d_K) \in \{1,2,...,\sqrt[K]{b}\}^K$. Here $b\ge 1$ is an integer that depends on $N,n,p$ and some constant $c_0$, and will be specified later. $\varphi(u)$ is defined in \eqref{eq::univar_func}. Then, we have the following technical lemma for $\Phi_{\bm d}$, which will later be used for constructing the packing set.

\begin{lemma}\label{lem::Phi_d}
	Suppose $\varphi(\cdot)$ are given by \eqref{eq::univar_func}. Then, $\Phi_{\bm d}$ has the following properties:
   \begin{enumerate}[label=(\roman*)]
		\item $\Phi_{\bm d} (\mathbf{x}) \in \mathcal{F}(L,\gamma,K)$.
		\item $\Phi_{\bm d}$ have disjoint support for different $\bm d$.
		\item There exist $C_{1,L,K}>0$ and $C_{2,L,K}>0$ only dependent on $L$ and $K$, for any given $\bm d$, $\Phi_{\bm d}$ has
		\[
		\gamma^2 C_{2,L,K} b^{-\frac{2L+K}{K}}  \le  \frac{1}{n}\sum_{\bm i=(i_1,...,i_K)}\Phi_{\bm d}^2(\bm \theta_{\bm i})   \le  \gamma^2  C_{1,L,K} b^{-\frac{2L+K}{K}}   
		\]
		when integer $b$ satisfies $1\le b \le 0.48^Kn$.
	\end{enumerate}
\end{lemma}

\begin{proof}
	For $\varphi(\cdot)$ in \eqref{eq::univar_func}, we have $|\frac{\partial^{l}}{\partial u^{l} }\varphi(u) | \le 1$, $\forall l=0,1,...,L$ such that 
	$\left|\frac{\partial^L}{\partial x_1^{L_1}\cdots x_K^{L_K}} \Phi_{\bm d} (\mathbf{x})\right| \leq \gamma$ holds for any $L_1+...+L_K=L$ for $\mathbf{x} \in \mathbb{R}^K_{[0,1]}$. Thus, $\Phi_{\bm d} (\mathbf{x}) \in \mathcal{F}(L,\gamma,K)$.

	Given that $\varphi(u)>0$ if and only if $u\in(-1/2,1/2)$, we have $\varphi_{d_k}(x) \equiv \varphi(\sqrt[K]{b}x-d_k+1/2)>0$ if and only if $x \in\left(\frac{d_k-1}{\sqrt[K]{b}},\frac{d_k}{\sqrt[K]{b}}\right)$ for $d_k \in \{1,...,\sqrt[K]{b}\}$. So, for each, we can divide the space $[0,1]$ into $\sqrt[K]{b}$ intervals, i.e.,
	\[
	\Delta_{1} = \left[0, \frac{1}{\sqrt[K]{b}}\right],~~\Delta_{d_k} = \left(\frac{d_k-1}{\sqrt[K]{b}},\frac{d_k}{\sqrt[K]{b}}\right], ~d_k=2,...,\sqrt[K]{b},
	\]
	such that $\Delta_{d_k} \cap \Delta_{d_k'}=\emptyset$ for $d_k\neq d_k'$ and $\cup_{d_k}\Delta_{d_k} = [0,1]$. Thus, $\varphi_{d_k}(x)$ have disjoint support and their support union is the unit interval. 
	
	Because $\Phi_{\bm d}$ is the product of $\varphi_{d_k}$, they also have disjoint supports. That is, for each $\bm{d}$, $\Phi_{\bm d}(\mathbf{x}) >0$ only when  $\mathbf{x} \in \Delta_{\bm d}$ where
	\[
	\begin{aligned}
		\Delta_{\bm d = (1,1,...,1)} &= \left[0, \frac{1}{\sqrt[K]{b}}\right]\times...\times  \left[0, \frac{1}{\sqrt[K]{b}}\right],\\
		\Delta_{\bm d=(d_1,...d_K)} & = \left(\frac{d_1-1}{\sqrt[K]{b}},\frac{d_1}{\sqrt[K]{b}}\right]\times ...\times \left(\frac{d_K-1}{\sqrt[K]{b}},\frac{d_K}{\sqrt[K]{b}}\right],
	\end{aligned}
	\] 
	$d_k=2,...,\sqrt[K]{b}$ for $k=1,...,K$, such that $\Delta_{\bm d} \cap \Delta_{\bm d'}=\emptyset$ if $\bm d \neq \bm d'$ and $\cup_{\bm d}\Delta_{\bm d}=[0,1]^K$. So, the space $[0,1]^K$ is divided into $b$ disjoint cubes.

	As for $(iii)$,  we know there exists a constant $c_L$ that only depends on $L$ such that $\sup_{u} |\varphi(u) |= \varphi(0) = c_L$, and $\|\varphi\|_2^2 > 0.49 c_L^2$. Then
	
	\begin{equation}\label{eq::upper_Phi_d}
		\begin{aligned}
			\frac{1}{n}\sum_{\bm i=(i_1,...,i_K)}\Phi_{\bm d}^2(\bm \theta_{\bm i}) 
			& = \frac{1}{n}\gamma^2b^{-2L/K}\sum_{ i_1=1}^{\sqrt[K]{b}}\cdots \sum_{ i_K=1}^{\sqrt[K]{b}}  \prod_{k=1}^{K} \varphi^2\left(\sqrt[K]{\frac{b}{n}}i_k-d_k+1/2\right)\\
			& = \frac{1}{n}\gamma^2b^{-2L/K}\sum_{ i_2=1}^{\sqrt[K]{b}}\cdots \sum_{ i_K=1}^{\sqrt[K]{b}}  \prod_{k=2}^{K} \varphi^2\left(\sqrt[K]{\frac{b}{n}}i_k-d_k+1/2\right) \\
			& ~~~~~~~~~~~~~~~~~~~~~~~~~~~~\times \left[\sum_{ i_1=1}^{\sqrt[K]{b}}\varphi^2\left(\sqrt[K]{\frac{b}{n}}i_1-d_1+1/2\right)\right]\\
			& = \frac{1}{n}\gamma^2b^{-2L/K}\sum_{ i_3=1}^{\sqrt[K]{b}}\cdots \sum_{ i_K=1}^{\sqrt[K]{b}}  \prod_{k=3}^K\varphi^2\left(\sqrt[K]{\frac{b}{n}}i_k-d_k+1/2\right) \\
			& ~~~~~~\times \left[\sum_{ i_1=1}^{\sqrt[K]{b}}\varphi^2\left(\sqrt[K]{\frac{b}{n}}i_1-d_1+1/2\right)\right] \times \left[\sum_{ i_2=1}^{\sqrt[K]{b}}\varphi^2\left(\sqrt[K]{\frac{b}{n}}i_2-d_2+1/2\right)\right]\\
			& \cdots\\
			& = \frac{1}{n} \gamma^2b^{-2L/K}\prod_{k=1}^{K}\left\{ \sum_{ i_k=1}^{\sqrt[K]{n}} \varphi^2\left(\sqrt[K]{\frac{b}{n}} i_k- d_k+1/2\right) \right\}\\
			& = \frac{1}{n}  \gamma^2b^{-2L/K} \prod_{k=1}^{K}\left\{ \sum_{\sqrt[K]{\frac{n}{b}}(d_k-1) < i_k \le\sqrt[K]{\frac{n}{b} }d_k } \varphi^2\left(\sqrt[K]{\frac{b}{n}} i_k-d_k+1/2\right) \right\} \\
			& \le \frac{1}{n} \gamma^2b^{-2L/K} \prod_{k=1}^{K}\left\{ \sqrt[K]{\frac{n}{b}}  \times \varphi^2(0)\right\} \\
			& =  \gamma^2b^{-\frac{2L+K}{K}}   c_L^{2K}. 
		\end{aligned}
	\end{equation}
	Therefore, $\frac{1}{n}\sum_{\bm i=(i_1,...,i_K)}\Phi_{\bm d}^2(\bm \theta_{\bm i}) \le c_L^{2K} \gamma^2b^{-\frac{2L+K}{K}}$ and $c_{L}^{2K}$ is the constant we find for $C_{1,L,K}$.
	
	On the other hand, we use the fact that the upper Riemann sum is greater than the integral of the function. Thus, for each coordinate,
	\[
	\begin{aligned}
		\sqrt[K]{\frac{b}{n}} \sum_{ i_k=1}^{\sqrt[K]{n}} \varphi^2\left(\sqrt[K]{\frac{b}{n}} i_k- d_k+1/2\right) 
		& =  \sqrt[K]{\frac{b}{n}} \sum_{\sqrt[K]{\frac{n}{b}}(d_k-1) < i_k \le\sqrt[K]{\frac{n}{b} }d_k } \varphi^2 \left[ 	\sqrt[K]{\frac{b}{n}} \left( i_k - \frac{\sqrt[K]{n} (d_k - 1/2)}{\sqrt[K]{b}}\right)   \right] \\
		& = \sqrt[K]{\frac{b}{n}}  \sum_{\frac{-\sqrt[K]{n}}{2\sqrt[K]{b}} < t \le \frac{\sqrt[K]{n}}{2\sqrt[K]{b}}}\varphi^2\left(\sqrt[K]{b/n}\times t\right) \\	\underrightarrow{\left[\varphi(-u)=\varphi(u)\right]} 
		& = 	\sqrt[K]{\frac{b}{n}}  \left[ 2\sum_{0\le t \le \frac{\sqrt[K]{n}}{2\sqrt[K]{b}}}\varphi^2(\sqrt[K]{b/n}\times t) - \varphi^2(0)\right]\\
		\underrightarrow{\left[\varphi(u) \text{ decreases for } u\ge0\right]}
		& \ge  2\sqrt[K]{\frac{b}{n}}  \int_0^\frac{\sqrt[K]{n}}{2\sqrt[K]{b}} \varphi^2(\sqrt[K]{b/n}\times t) dt - \sqrt[K]{\frac{b}{n}} \varphi^2(0)\\
		& = 2  \int_0^{1/2} \varphi^2(u) du - \sqrt[K]{\frac{b}{n}} c_L^2\\
		& = \|\varphi \|_2^2 - \sqrt[K]{\frac{b}{n}} c_L^2\\
		& > c_L^2\left(0.49 - \sqrt[K]{b/n} \right)\\
		\underrightarrow{\left[1\le b \le  0.49^{K}n\right]}
		& \ge 0,
	\end{aligned}
	\]

	Thus, the empirical sum can also be lower bounded by
	\begin{equation}
		\begin{aligned}
			\frac{1}{n}\sum_{\bm i=(i_1,...,i_K)}\Phi_{\bm d}^2(\bm \theta_{\bm i})
			&= \gamma^2b^{-\frac{2L+K}{K}}\prod_{k=1}^{K}\left\{  \sqrt[K]{\frac{b}{n}}\sum_{ i_k=1}^{\sqrt[K]{n}} \varphi^2\left(\sqrt[K]{\frac{b}{n}} i_k- d_k+1/2\right) \right\}\\
			& \ge \gamma^2b^{-\frac{2L+K}{K}}c_L^{2K}\left(0.49 - \sqrt[K]{b/n} \right)^K 
		\end{aligned}
	\end{equation}

	When $1\le b\le 0.48^Kn$, 
	\begin{equation}\label{eq::lower_Phi_d}
		\frac{1}{n}\sum_{\bm i=(i_1,...,i_K)}\Phi_{\bm d}^2(\bm \theta_{\bm i}) \ge \gamma^2(0.1c_L)^{2K} b^{-\frac{2L+K}{K}},
	\end{equation}
	and thus $(0.1c_{L})^{2K}$ is the constant we find for $C_{2,L,K}$.
	Combining \eqref{eq::upper_Phi_d} and \eqref{eq::lower_Phi_d}, we have
	\[
	\gamma^2 C_{2,L,K} b^{-\frac{2L+K}{K}}  \le  \frac{1}{n}\sum_{\bm i=(i_1,...,i_K)}\Phi_{\bm d}^2(\bm \theta_{\bm i})   \le  \gamma^2  C_{1,L,K} b^{-\frac{2L+K}{K}}   
	\]
	as claimed.
\end{proof}

In proving the lower bound, we shall use Fano's method (see Section 15.3.2 in \cite{wainwright2019high}). To do so, we first establish the connection between minimax risks and error probabilities in testing problems (for completeness), and then apply Fano's inequality to lower bound the error probabilities. To this end, we first provide the following lemma, which shows that there exists a packing set of hypotheses with suitably large cardinality, for which the mutual information (stated in terms of Kullback-Leibler divergence) can be upper bounded. We can then use Fano's inequality with this set.

\begin{lemma}\label{lem::packing_set}
	Consider an arbitrary fixed $L$, $\gamma$ and $K$. For some constant $C_{1,L,K}$ and $C_{2,L,K}$  that only depends on $L$ and $K$, and for some other constant $c_0>0$, there exists a subset $\mathcal{B}^0 \subseteq \mathcal{M}(L,\gamma,K)$ with cardinality
	\[
	|\mathcal{B}^0| \ge 2^{\lceil c_0 \left(\frac{n\vee p}{N}\right)^{\frac{-K}{2L+K}}\rceil \times p/8} + 1,
	\]
	when $p\ge 8$, that has the following properties:
	\begin{enumerate}[label=(\roman*)]
		\item $\mathcal{B}^0$ is a $2\delta_{N,n,p}$-packing set, i.e., for any $M_{s} \neq M_{s'} \in \mathcal{B}^0$, 
		\[
		\frac{1}{np}\|M_{s} - M_{s'}\|_F^2 \ge 2\delta_{N,n,p} =  \frac{C_{2,L,K}  \gamma^2}{8}  (2c_0)^{-2L/K} \left(\frac{n\vee p}{N}\right)^{\frac{2L}{2L+K}}
		\]
		when $ c_0^{-\frac{2L+K}{K}}(n\vee p) \le N \le c_0^{-\frac{2L+K}{K}} 0.48^{2L+K} (n\vee p)n^{\frac{2L+K}{K}}$. 
		\item For any $M_{s},  M_{s'} \in \mathcal{B}^0$,
		\[
		K(\mathbb{P}_s||\mathbb{P}_{s'}) \le \frac{C_{1, L,K} \gamma^2}{2\sigma^2} c_0^{-\frac{2L}{K}}  N \left(\frac{n\vee p}{N}\right)^{\frac{2L}{2L+K}} 
		\]
		where $K(\mathbb{P}_s||\mathbb{P}_{s'})$ denotes the Kullback-Leibler divergence between probability distributions of observations $\{(y_t,X_t)\}_{t=1}^N$ satisfying model \eqref{eq::supp_data_model}, given $M_{s}$ and $M_{s'} $ respectively. 
	\end{enumerate}
\end{lemma}

\begin{proof}

	We will consider a positive integer $b$ which depends on $N, n, p$ and a constant $c_0$. The precise specification of $b$ will come later.  Consider the multivariate function $\Phi_{\bm d}(\bm{\theta})$ in \eqref{eq::multivar_func}.
	
	We will define a set $\Omega$ that is used to construct packing matrices where each element $\omega$ in $\Omega$ is a sequence (of length $b$) of diagonal matrices. We index the set in a somewhat curious way: We use a multi-index of dimension $K$ where each index has elements in $\{1,...,\sqrt[K]{b}\}$. This will ease exposition later.
	\begin{equation}
		\begin{aligned}
			\bm \Omega = \left\{\bm w  = (\bm w_{\bm d})_{\bm{d}\in \{1,...,\sqrt[K]{b}\}^K}:
			\text{for each }\bm{d},~\bm w_{\bm d}= \operatorname{diag}(w_{\bm d,1} ,...,w_{\bm d,p} ),
			w_{\bm d,j} \in \{0,1\} \right \}.
		\end{aligned}
	\end{equation}
	From this we define the following collection of matrices,
	\begin{equation}\label{eq::hypothesis_B}
		\begin{aligned}
			\mathcal{B} 
			& = \left\{M_{\bm{w}} = \sum_{d_1=1}^{\sqrt[K]{b}} \ldots\sum_{d_K=1}^{\sqrt[K]{b}} \begin{pmatrix}
				\Phi_{\bm d}(\bm{\theta}_1)w_{\bm d, 1} & \Phi_{\bm d}(\bm{\theta}_1) w_{\bm d, 2} & \ldots &\Phi_{\bm d}(\bm{\theta}_1) w_{\bm d, p}\\
				\Phi_{\bm d}(\bm{\theta}_2) w_{\bm d, 1}& \Phi_{\bm d}(\bm{\theta}_2) w_{\bm d, 2} & \ldots &\Phi_{\bm d}(\bm{\theta}_2) w_{\bm d, p}\\
				\vdots & \vdots & \ddots & \vdots\\
				\Phi_{\bm d}(\bm{\theta}_n)w_{\bm d, 1} & \Phi_{\bm d}(\bm{\theta}_n)  w_{\bm d, 2}& \ldots &\Phi_{\bm d}(\bm{\theta}_n)w_{\bm d, p} \\
			\end{pmatrix}_{n\times p}, ~~w_{\bm d, j} \in \{0,1\} \right\}\\
			& =: \left\{ M_{\bm w} = \sum_{d_1,...,d_K}\Phi_{\bm d}(\bm \Theta)\bm w_{\bm d}, \text{ for } \bm w = (\bm w_{\bm d})\in \bm \Omega  \right\}.
		\end{aligned}
	\end{equation}
	We see that we can compactly write each matrix in our set as the product of $\Phi_{\bm d}(\bm \Theta)$ and $\bm w_{\bm d}$, where $\Phi_{\bm d}(\bm \Theta)$ is a $n\times p$ matrix whose elements in the $i$-th row are all $\Phi_{\bm d}(\bm{\theta}_i)$. It is direct to check that the cardinality of $\Omega$ is given by $|\bm \Omega| = |\mathcal{B}|=2^{bp}$.

	Thus, entries of $M_{\bm w} \in \mathcal{B} $ can be written as $m_{ij} = \sum_{d_1,...,d_K}\Phi_{\bm d}(\bm{\theta}_i)w_{\bm d, j} = g_j(\bm \theta_i)$, where $g_j$ has bounded derivatives, 
	\[
	\left|\frac{\partial^L g_j(\mathbf{x})}{\partial x_1^{L_1}\cdots x_K^{L_K}} \right| \le \sum_{d_1,...,d_K} \left| \frac{\partial^L \Phi_{\bm d}(\mathbf{x}) }{\partial x_1^{L_1}\cdots x_K^{L_K}} \right| =   \left| \frac{\partial^L \Phi_{\bm d}(\mathbf{x}) }{\partial x_1^{L_1}\cdots x_K^{L_K}} \right| \mathbb{1}\{\mathbf{x} \in  \Delta_{\bm d}\}\le \gamma
	\]
	for $\forall \mathbf{x} \in \mathbb{R}^K_{[0,1]}$. Hence, $\mathcal{B} \subseteq \mathcal{M}(L,\gamma,K)$.

	Consider a set of testing hypotheses from $\mathcal{B}$, 
	\begin{equation}\label{eq::testing_set}
		\mathcal{B}^0 = \{M_{\bm w^{(0)}},...,M_{\bm w^{(S)}}\} \subseteq \mathcal{B},~~\bm w^{(s)} \in \bm \Omega, ~s=0,1,...,S,
	\end{equation}
	where $\bm w^{(s)} \neq \bm w^{(s')}$ for $0\le s \neq s' \le S$.  
	
	For any $0\le s \neq s' \le S$, and constant $C_{2,L,K}$ only dependent on $L$, 
	\begin{equation}\label{eq::2-delta}
		\begin{aligned}
			\frac{1}{np}\|M_{\bm w^{(s)}} - M_{\bm w^{(s')}}\|_F^2 
			& = \frac{1}{np}\sum_{\bm i=(i_1,...,i_K)}\sum_{j=1}^{p} \left[\sum_{d_1=1}^{\sqrt[K]{b}} \ldots\sum_{d_K=1}^{\sqrt[K]{b}}(w^{(s)}_{\bm d,j} - w^{(s')}_{\bm d, j}) \Phi_{\bm d}(\bm \theta_{\bm i})\right]^2\\
			\underrightarrow{\text{the support of } \Phi_{\bm d} \text{'s are disjoint}}& =  \frac{1}{p}\sum_{j=1}^{p} \sum_{d_1=1}^{\sqrt[K]{b}} \ldots\sum_{d_K=1}^{\sqrt[K]{b}}(w^{(s)}_{\bm d,j} - w^{(s')}_{\bm d, j})^2 \left(\frac{1}{n}\sum_{\bm i=(i_1,...,i_K)}\Phi_{\bm d}^2(\bm \theta_{\bm i}) \right) \\
			\underrightarrow{\text{Lemma~\ref{lem::Phi_d}-(i)}}
			&  \ge \gamma^2C_{2,L,K}  b^{-\frac{2L+K}{K}}  p^{-1} \rho(\bm w^{(s)}, \bm w^{(s')})
		\end{aligned}
	\end{equation}
	where $\rho(\bm w^{(s)}, \bm w^{(s')})= \sum_{j=1}^{p} \sum_{d_1=1}^{\sqrt[K]{b}} \ldots\sum_{d_K=1}^{\sqrt[K]{b}}(w^{(s)}_{\bm d,j} - w^{(s')}_{\bm d, j})^2$ is the hamming distance between $\bm w^{(s)}$ and $\bm w^{(s')}$.
	
	Due to the Varshamov–Gilbert bound (Lemma~2.9 in \cite{tsybakov2009introduction}), when $bp \ge 8$, there exists a subset $\bm\Omega^0 = (\bm w^{(0)}, ..., \bm w^{(S)} )  \subseteq \bm \Omega$ such that $S \ge 2^{bp/8}$ and $\rho(\bm w^{(s)},\bm w^{(s')}) \ge bp/8$ for $0 \le s \neq s' \le S$. Since $b\ge 1$, $p \ge 8$ is a sufficient condition to guarantee $bp \ge 8$. 
	
	Now, in particular, we choose our testing set based on $\bm\Omega^0$: That is, we place $M_{\bm w^{(s)}} \in \mathcal{B}^0$ if and only if $\bm{w}^{(s)} \in \bm{\Omega}^0$. In particular this gives us that $\rho(\bm w^{(s)},\bm w^{(s')}) \ge bp/8$. for all $w^{(s)}, w^{(s')} \in \mathcal{B}^0$ with  $s \neq s'$.
	Then, following \eqref{eq::2-delta}, we have that
	\begin{equation}\label{eq::2-delta-2}
		\frac{1}{np}\|M_{\bm w^{(s)}} - M_{\bm w^{(s')}}\|_F^2   \ge \frac{\gamma^2C_{2,L,K} }{8} b^{-\frac{2L}{K}}.
	\end{equation}

	\emph{Now, we finally give the value that we use for $b$}: Select $b=\left\lceil c_0 \left(\frac{n\vee p}{N}\right)^{\frac{-K}{2L+K}}\right\rceil$ for some constant $c_0>0$. We note that \eqref{eq::2-delta}-\eqref{eq::2-delta-2} hold only when $b \le 0.48^K n $ as stated in Lemma~\ref{lem::Phi_d}. So, we need
	\begin{equation}\label{eq::N_bound}
		N \le c_0^{-\frac{2L+K}{K}} 0.48^{2L+K} (n\vee p)n^{\frac{2L+K}{K}}.
	\end{equation}
	Furthermore, we also need 
	\begin{equation}
		N\ge c_0^{-\frac{2L+K}{K}}(n\vee p)  
	\end{equation}
	such that
	\begin{equation}
		b=\left\lceil c_0 \left(\frac{n\vee p}{N}\right)^{\frac{-K}{2L+K}}\right\rceil\le 2c_0 \left(\frac{n\vee p}{N}\right)^{\frac{-K}{2L+K}}.
	\end{equation}
	This, finally gives us
	%
	\begin{equation}\label{eq::2-delta-3}
		\frac{1}{np}\|M_{\bm w^{(s)}} - M_{\bm w^{(s')}}\|_F^2   \ge  \frac{ C_{2,L,K} \gamma^2}{8}  (2c_0)^{-2L/K} \left(\frac{n\vee p}{N}\right)^{\frac{2L}{2L+K}}=: 2\delta_{N,n,p}.
	\end{equation}
	
	Then $\mathcal{B}^0$ is a $2\delta_{N,n,p}$-packing set of $\mathcal{M}(L,\gamma, K)$ and the cardinality $|\mathcal{B}^0| = S +1 \ge 2^{bp/8} + 1 = 2^{\lceil c_0 \left(\frac{n\vee p}{N}\right)^{\frac{-K}{2L+K}}\rceil \times p/8}  + 1$ when $p\ge 8$.

	We now show the second property (related to the KL distance) of $\mathcal{B}^0$. For any matrices $M_{\bm w^{(s)}}, M_{\bm w^{(s')}} \in \mathcal{B}^0$, with the selected $b=\lceil c_0 \left(\frac{n\vee p}{N}\right)^{\frac{-K}{2L+K}}\rceil$, we have
	\begin{equation}\hspace{-1.3cm}
		\begin{aligned}
			K(\mathbb{P}_s||\mathbb{P}_{s'})  & = \int \log \frac{d\mathbb{P}_{s}}{d\mathbb{P}_{s'}}d\mathbb{P}_{s}  \\
			& =  \int\int\log\frac{\prod_{t=1}^{N}p(y_t, X_t|M_{w^{(s)}})}{\prod_{t=1}^{N}p(y_t, X_t|M_{w^{(s')}})}  \left[\prod_{t=1}^{N}p(y_t, X_t|M_{w^{(s)}})dy_tdX_t\right] \\
			\underrightarrow{[\text{Bayes' Rule}]}& = \operatorname{E}_{X\sim \Pi} \sum_{t=1}^{N}  \int \left[\log p(y_t|X_t, M_{w^{(s)}})  -\log p(y_t|X_t, M_{w^{(s')}})\right] p(y_t|X_t, M_{w^{(s)}}) dy_t\\
			\underrightarrow{(y_t|X_t, M) \sim_{i.i.d.} N[\langle X_t, M\rangle, \sigma^2]}	& =  \operatorname{E}_{X\sim \Pi} \sum_{t=1}^{N} \frac{\left\langle X_t, M_{w^{(s)}} - M_{w^{(s')}} \right\rangle^2}{2\sigma^2} \\
			\underrightarrow{\left[E_{X\sim \Pi}\langle X_t,M\rangle^2 = \frac{1}{np}\|M\|_F^2\right]}
			& = \frac{N}{2\sigma^2np} \|M_{\bm w^{(s)}} - M_{\bm w^{(s')}}\|_F^2\\
			&\le \frac{N}{2\sigma^2} \sum_{d_1=1}^{\sqrt[K]{b}} \ldots\sum_{d_K=1}^{\sqrt[K]{b}} \left(\frac{1}{n}\sum_{\bm i=(i_1,...,i_K)}\Phi_{\bm d}^2(\bm \theta_{\bm i}) \right)\\
			& \le \frac{N\gamma^2}{2\sigma^2} b^{-\frac{2L}{K}} C_{1,L,K}\\
			& \le \frac{C_{1,L,K} \gamma^2}{2\sigma^2} c_0^{-\frac{2L}{K}}  N \left(\frac{n\vee p}{N}\right)^{\frac{2L}{2L+K}}. 
		\end{aligned}
	\end{equation}
	Thus, Lemma~\ref{lem::packing_set} is proved.
\end{proof}

\subsection{Information-theoretic lower bounds}

Given Lemma~\ref{lem::packing_set}, we now apply the argument in \cite{yang1999information} to yield a lower bound for error in our estimation problem with respect to Frobenius norm.

\begin{proof}[Proof of Theorem~\ref{thm::lower_bound}]

	For a given $\delta_{N,n,p}$, let $\mathcal{B}^0$ be the $2\delta_{N,n,p}$-packing set of $\mathcal{M}(L,\gamma, K)$ indicated by Lemma~\ref{lem::packing_set}. We know that for any $M_{s} \neq M_{s'} \in \mathcal{B}^0$, 
	\[
	\frac{1}{np}\|M_{s} - M_{s'}\|_F^2 \ge 2\delta_{N,n,p} 
	\]
	with $\delta_{N,n,p}  = \frac{C_{2,L,K}  \gamma^2}{16}  (2c_0)^{-2L/K} \left(\frac{n\vee p}{N}\right)^{\frac{2L}{2L+K}}$, when $c_0^{-\frac{2L+K}{K}}(n\vee p) \le N \le c_0^{-\frac{2L+K}{K}} 0.48^{2L+K} (n\vee p)n^{\frac{2L+K}{K}}$ for some constant $c_0 > 0$.

	Let $d(M_1, M_2) = \frac{1}{np}\|M_1-M_2\|_F^2$ and define 
	$$\tilde{M} = \arg\min_{M' \in \mathcal{B}^0} d(M', \widehat{M})\in \mathcal{B}^0.$$
	Let $M$ be any matrix in the packing set $\mathcal{B}^0$. If $d(M, \widehat{M}) < \delta_{N,n,p}$, then $\max\left\{d(M, \widehat{M}), d(\tilde{M}, \widehat{M})\right\} = d(M, \widehat{M}) < \delta_{N,n,p} \le \delta_0 \equiv C_{2,L,K}  \gamma^2  4^{-(L+2K)/K} $.
	Then, by the triangle inequality, we have $d(M, \widehat{M}) + d(\tilde{M}, \widehat{M}) \ge d(M, \tilde{M}) \ge 2\delta_{N,n,p}$ when $M\neq \tilde{M}$. 
	This implies that $d(M, \widehat{M}) \ge \delta_{N,n,p}$, which contradicts $d(M, \widehat{M}) < \delta_{N,n,p}$. Therefore, if $M \neq \tilde{M}$, we must have $d(M, \widehat{M} )\ge \delta_{N,n,p}$. So, it follows that
	\begin{equation}\label{eq::minimax}
		\begin{aligned}
			\inf_{\widehat M}\sup_{M\in \mathcal{M}(L,\gamma, K)}\mathbb{P}\left\{ d(M, \widehat{M}) \ge \delta_{N,n,p}\right\} 
			&\ge \inf_{\widehat M}\sup_{M\in \mathcal{B}^0}\mathbb{P}\left\{d(M, \widehat{M}) \ge \delta_{N,n,p}\right\}\\
			&  =  \inf_{\widehat M}\sup_{M\in \mathcal{B}^0}\mathbb{P}\left\{ M \neq \tilde{M} \right\}\\
			& \ge  \inf_{\widehat{M}} \mathbb{P}(M \neq \tilde{M})
		\end{aligned}
	\end{equation}
	where $M$ is uniformly distributed over the $2\delta_{N,n,p}$-packing set $\mathcal{B}^0$ with $|\mathcal{B}^0| \ge 2^{\lceil c_0 \left(\frac{n\vee p}{N}\right)^{\frac{-K}{2L+K}}\rceil \times p/8} + 1$ as in Lemma~\ref{lem::packing_set}. this has reduced our problem essentially to a testing problem.
	
	We now use this to obtain a lower bound, by considering KL-divergence here. By Lemma~\ref{lem::packing_set}(iii), Fano's inequality \citep{cover2012elements}, or \citep[Proposition 15.12]{wainwright2019high}, and the convexity of the Kullback–Leibler divergence \citep[(15.34)]{wainwright2019high}, 
	\begin{equation}\label{eq::prob_bound0}
		\begin{aligned}
			\mathbb{P}(M \neq \tilde{M})  
			&\ge 1-\frac{\frac{1}{|\mathcal{B}^0|^2}\sum_{M_s, M_{s'}\in\mathcal{B}^0}  K(\mathbb{P}_s||\mathbb{P}_{s'})+ \log 2}{\log |\mathcal{B}^0|}\\
			\underrightarrow{[\text{Lemma}~\ref{lem::packing_set}]}  		 
			& \ge 1- \frac{ \frac{C_{1,L,K} \gamma^2}{2\sigma^2} c_0^{-\frac{2L}{K}}  N \left(\frac{n\vee p}{N}\right)^{\frac{2L}{2L+K}} + \log 2}{\lceil c_0 \left(\frac{n\vee p}{N}\right)^{\frac{-K}{2L+K}}\rceil p\log 2} \\
			\underrightarrow{[bp\ge 8]}& \ge \frac{7}{8} - \frac{C_{1,L,K} \gamma^2 c_0^{-\frac{2L+K}{K}} (n\vee p) }{2(\log 2)\sigma^2 p}.
		\end{aligned}
	\end{equation}

	Consider $n = \kappa p$ for some $\kappa>0$. Let
	\begin{equation}\label{eq::c_0}
		c_0= \left(\frac{4\max(\kappa,1)\gamma^2 C_{1,L,K}}{3\log2\sigma^2}\right)^{\frac{K}{2L+K}}.
	\end{equation}
	Then,
	\begin{equation}\label{eq::prob_bound2}.
		\begin{aligned}
			\mathbb{P}(M \neq \tilde{M}) 
			& \ge 7/8- \frac{\gamma^2 C_{1,L,K} \max(\kappa,1) c_0^{-\frac{2L+K}{K}}}{2(\log 2)\sigma^2 } = 7/8-3/8=1/2
		\end{aligned}
	\end{equation}

	Thus, it follows from \eqref{eq::minimax} and \eqref{eq::prob_bound2} that
	\begin{equation}\label{eq::supp_lower_bound}
		\inf_{\widehat M}\sup_{M\in \mathcal{M}(L,\gamma, K)}\mathbb{P}\left\{\frac{1}{np}\|\widehat M-M\|_F^2 
		\ge 
		A \left(\frac{n\vee p}{N}\right)^{\frac{2L}{2L+K}} \right\} \ge  1/2.
	\end{equation}
	where $A= \frac{C_{2,L,K}  \gamma^2}{16}  (2c_0)^{-2L/K}$. With the selection of $c_0$ in \eqref{eq::c_0}, $A$ depends on $L,K,\gamma,\kappa,\sigma^2$.
	
	Thus, Theorem~\ref{thm::lower_bound} is proved.

\end{proof}

\bibliographystyle{chicago}      
\bibliography{bibfile}   

\begin{thebibliography}{}

\bibitem[\protect\citeauthoryear{Argyriou, Evgeniou, and Pontil}{Argyriou
  et~al.}{2008}]{argyriou2008convex}
Argyriou, A., T.~Evgeniou, and M.~Pontil (2008).
\newblock Convex multi-task feature learning.
\newblock {\em Machine Learning\/}~{\em 73\/}(3), 243--272.

\bibitem[\protect\citeauthoryear{Bach}{Bach}{2008}]{bach2008consistency}
Bach, F.~R. (2008).
\newblock Consistency of trace norm minimization.
\newblock {\em Journal of Machine Learning Research\/}~{\em 9\/}(Jun),
  1019--1048.

\bibitem[\protect\citeauthoryear{Cai, Zhou, et~al.}{Cai
  et~al.}{2016}]{cai2016matrix}
Cai, T.~T., W.-X. Zhou, et~al. (2016).
\newblock Matrix completion via max-norm constrained optimization.
\newblock {\em Electronic Journal of Statistics\/}~{\em 10\/}(1), 1493--1525.

\bibitem[\protect\citeauthoryear{Candes and Plan}{Candes and
  Plan}{2010a}]{CandesPlan2009}
Candes, E.~J. and Y.~Plan (2010a, June).
\newblock Matrix completion with noise.
\newblock {\em Proceedings of the IEEE\/}~{\em 98\/}(6), 925--936.

\bibitem[\protect\citeauthoryear{Candes and Plan}{Candes and
  Plan}{2010b}]{candes2010matrix}
Candes, E.~J. and Y.~Plan (2010b).
\newblock Matrix completion with noise.
\newblock {\em Proceedings of the IEEE\/}~{\em 98\/}(6), 925--936.

\bibitem[\protect\citeauthoryear{Candes and Plan}{Candes and
  Plan}{2011}]{candes2011tight}
Candes, E.~J. and Y.~Plan (2011).
\newblock Tight oracle inequalities for low-rank matrix recovery from a minimal
  number of noisy random measurements.
\newblock {\em IEEE Transactions on Information Theory\/}~{\em 57\/}(4),
  2342--2359.

\bibitem[\protect\citeauthoryear{Candes and Tao}{Candes and
  Tao}{2010}]{CandesTao2010}
Candes, E.~J. and T.~Tao (2010, May).
\newblock The power of convex relaxation: Near-optimal matrix completion.
\newblock {\em IEEE Transactions on Information Theory\/}~{\em 56\/}(5),
  2053--2080.

\bibitem[\protect\citeauthoryear{Chatterjee et~al.}{Chatterjee
  et~al.}{2015}]{chatterjee2015matrix}
Chatterjee, S. et~al. (2015).
\newblock Matrix estimation by universal singular value thresholding.
\newblock {\em The Annals of Statistics\/}~{\em 43\/}(1), 177--214.

\bibitem[\protect\citeauthoryear{Chistov and Grigoriev}{Chistov and
  Grigoriev}{1984}]{Chistov1984ComplexityOQ}
Chistov, A.~L. and D.~Grigoriev (1984).
\newblock Complexity of quantifier elimination in the theory of algebraically
  closed fields.
\newblock In {\em MFCS}.

\bibitem[\protect\citeauthoryear{Cover and Thomas}{Cover and
  Thomas}{2012}]{cover2012elements}
Cover, T.~M. and J.~A. Thomas (2012).
\newblock {\em Elements of Information Theory}.
\newblock John Wiley \& Sons.

\bibitem[\protect\citeauthoryear{Fan and Cheng}{Fan and Cheng}{2018}]{Fan2018b}
Fan, J. and J.~Cheng (2018).
\newblock Matrix completion by deep matrix factorization.
\newblock {\em Neural Networks\/}~{\em 98}, 34 -- 41.

\bibitem[\protect\citeauthoryear{Fan and Chow}{Fan and Chow}{2018}]{Fan2018a}
Fan, J. and T.~Chow (2018, 5).
\newblock Non-linear matrix completion.
\newblock {\em Pattern Recognition\/}~{\em 77}, 378--394.

\bibitem[\protect\citeauthoryear{Fazel}{Fazel}{2002}]{fazel2002}
Fazel, M. (2002).
\newblock {\em Matrix rank minimization with applications}.
\newblock Ph.\ D. thesis, PhD thesis, Stanford University.

\bibitem[\protect\citeauthoryear{Jaggi and Sulovsk{\`y}}{Jaggi and
  Sulovsk{\`y}}{2010}]{jaggi2010simple}
Jaggi, M. and M.~Sulovsk{\`y} (2010).
\newblock A simple algorithm for nuclear norm regularized problems.
\newblock In {\em ICML}.

\bibitem[\protect\citeauthoryear{Jiang, Zheng, Tan, Tang, and Zhou}{Jiang
  et~al.}{2016}]{jiang2016variational}
Jiang, Z., Y.~Zheng, H.~Tan, B.~Tang, and H.~Zhou (2016).
\newblock Variational deep embedding: An unsupervised and generative approach
  to clustering.
\newblock {\em arXiv preprint arXiv:1611.05148\/}.

\bibitem[\protect\citeauthoryear{Klopp et~al.}{Klopp
  et~al.}{2014}]{klopp2014noisy}
Klopp, O. et~al. (2014).
\newblock Noisy low-rank matrix completion with general sampling distribution.
\newblock {\em Bernoulli\/}~{\em 20\/}(1), 282--303.

\bibitem[\protect\citeauthoryear{Koltchinskii, Lounici, and
  Tsybakov}{Koltchinskii et~al.}{2011}]{koltchinskii2011}
Koltchinskii, V., K.~Lounici, and A.~B. Tsybakov (2011, 10).
\newblock Nuclear-norm penalization and optimal rates for noisy low-rank matrix
  completion.
\newblock {\em Ann. Statist.\/}~{\em 39\/}(5), 2302--2329.

\bibitem[\protect\citeauthoryear{Koren, Bell, and Volinsky}{Koren
  et~al.}{2009}]{Koren2009}
Koren, Y., R.~Bell, and C.~Volinsky (2009, August).
\newblock Matrix factorization techniques for recommender systems.
\newblock {\em Computer\/}~{\em 42\/}(8), 30--37.

\bibitem[\protect\citeauthoryear{Laurent}{Laurent}{2001}]{Laurent2001}
Laurent, M. (2001).
\newblock {\em Matrix completion problems}, pp.\  221--229.
\newblock Netherlands: Kluwer Academic Publishers.
\newblock Pagination: 9.

\bibitem[\protect\citeauthoryear{Li, Shah, Song, and Yu}{Li
  et~al.}{2019}]{li2019nearest}
Li, Y., D.~Shah, D.~Song, and C.~L. Yu (2019).
\newblock Nearest neighbors for matrix estimation interpreted as blind
  regression for latent variable model.
\newblock {\em IEEE Transactions on Information Theory\/}~{\em 66\/}(3),
  1760--1784.

\bibitem[\protect\citeauthoryear{Liu and Vandenberghe}{Liu and
  Vandenberghe}{2010}]{Liu2010}
Liu, Z. and L.~Vandenberghe (2010).
\newblock Interior-point method for nuclear norm approximation with application
  to system identification.
\newblock {\em SIAM Journal on Matrix Analysis and Applications\/}~{\em
  31\/}(3), 1235--1256.

\bibitem[\protect\citeauthoryear{Negahban and Wainwright}{Negahban and
  Wainwright}{2011}]{negahban2011estimation}
Negahban, S. and M.~J. Wainwright (2011).
\newblock Estimation of (near) low-rank matrices with noise and
  high-dimensional scaling.
\newblock {\em The Annals of Statistics\/}, 1069--1097.

\bibitem[\protect\citeauthoryear{Recht}{Recht}{2011}]{Recht2011}
Recht, B. (2011, December).
\newblock A simpler approach to matrix completion.
\newblock {\em J. Mach. Learn. Res.\/}~{\em 12}, 3413--3430.

\bibitem[\protect\citeauthoryear{Recht, Fazel, and Parrilo}{Recht
  et~al.}{2010}]{Recht2010}
Recht, B., M.~Fazel, and P.~A. Parrilo (2010, August).
\newblock Guaranteed minimum-rank solutions of linear matrix equations via
  nuclear norm minimization.
\newblock {\em SIAM Rev.\/}~{\em 52\/}(3), 471--501.

\bibitem[\protect\citeauthoryear{Schreiber, Durham, Bilmes, and
  Noble}{Schreiber et~al.}{2018}]{Schreiber2018}
Schreiber, J., T.~J. Durham, J.~Bilmes, and W.~S. Noble (2018).
\newblock Multi-scale deep tensor factorization learns a latent representation
  of the human epigenome.
\newblock {\em bioRxiv\/}.

\bibitem[\protect\citeauthoryear{Singer and Cucuringu}{Singer and
  Cucuringu}{2010}]{singer2010uniqueness}
Singer, A. and M.~Cucuringu (2010).
\newblock Uniqueness of low-rank matrix completion by rigidity theory.
\newblock {\em SIAM Journal on Matrix Analysis and Applications\/}~{\em
  31\/}(4), 1621--1641.

\bibitem[\protect\citeauthoryear{Srebro, Rennie, and Jaakkola}{Srebro
  et~al.}{2004}]{Srebro2004}
Srebro, N., J.~D.~M. Rennie, and T.~S. Jaakkola (2004).
\newblock Maximum-margin matrix factorization.
\newblock In {\em Proceedings of the 17th International Conference on Neural
  Information Processing Systems}, NIPS'04, Cambridge, MA, USA, pp.\
  1329--1336. MIT Press.

\bibitem[\protect\citeauthoryear{Srebro and Shraibman}{Srebro and
  Shraibman}{2005}]{srebro2005rank}
Srebro, N. and A.~Shraibman (2005).
\newblock Rank, trace-norm and max-norm.
\newblock In {\em International Conference on Computational Learning Theory},
  pp.\  545--560. Springer.

\bibitem[\protect\citeauthoryear{Tsybakov}{Tsybakov}{2008}]{tsybakov2008introduction}
Tsybakov, A.~B. (2008).
\newblock {\em Introduction to nonparametric estimation}.
\newblock Springer Science \& Business Media.

\bibitem[\protect\citeauthoryear{Tsybakov}{Tsybakov}{2009}]{tsybakov2009introduction}
Tsybakov, A.~B. (2009).
\newblock Introduction to nonparametric estimation. revised and extended from
  the 2004 french original. translated by vladimir zaiats.

\bibitem[\protect\citeauthoryear{Udell and Townsend}{Udell and
  Townsend}{2019}]{udell2019big}
Udell, M. and A.~Townsend (2019).
\newblock Why are big data matrices approximately low rank?
\newblock {\em SIAM Journal on Mathematics of Data Science\/}~{\em 1\/}(1),
  144--160.

\bibitem[\protect\citeauthoryear{van~de Geer}{van~de
  Geer}{2016}]{van2016estimation}
van~de Geer, S. (2016).
\newblock Estimation and testing under sparsity.
\newblock {\em Lecture Notes in Mathematics\/}~{\em 2159}.

\bibitem[\protect\citeauthoryear{Wainwright}{Wainwright}{2019}]{wainwright2019high}
Wainwright, M.~J. (2019).
\newblock {\em High-dimensional statistics: A non-asymptotic viewpoint},
  Volume~48.
\newblock Cambridge University Press.

\bibitem[\protect\citeauthoryear{Wang, Chen, Ghosh, Denny, Kho, Chen, Malin,
  and Sun}{Wang et~al.}{2015}]{wang2015}
Wang, Y., R.~Chen, J.~Ghosh, J.~C. Denny, A.~Kho, Y.~Chen, B.~A. Malin, and
  J.~Sun (2015).
\newblock Rubik: Knowledge guided tensor factorization and completion for
  health data analytics.
\newblock In {\em Proceedings of the 21th ACM SIGKDD International Conference
  on Knowledge Discovery and Data Mining}, pp.\  1265--1274. ACM.

\bibitem[\protect\citeauthoryear{Watson}{Watson}{1992}]{watson1992characterization}
Watson, G.~A. (1992).
\newblock Characterization of the subdifferential of some matrix norms.
\newblock {\em Linear algebra and its applications\/}~{\em 170}, 33--45.

\bibitem[\protect\citeauthoryear{Wijaya, Callahan, Hewitt, Gao, Ling,
  Apidianaki, and Callison-Burch}{Wijaya et~al.}{2017}]{Wijaya2017}
Wijaya, D.~T., B.~Callahan, J.~Hewitt, J.~Gao, X.~Ling, M.~Apidianaki, and
  C.~Callison-Burch (2017).
\newblock Learning translations via matrix completion.
\newblock In {\em Proceedings of the 2017 Conference on Empirical Methods in
  Natural Language Processing}, pp.\  1452--1463. Association for Computational
  Linguistics.

\bibitem[\protect\citeauthoryear{Xia, Sun, Chen, Liu, Feng, Zhang, and
  Hang}{Xia et~al.}{2018}]{Xia2018}
Xia, G., H.~Sun, B.~Chen, Q.~Liu, L.~Feng, G.~Zhang, and R.~Hang (2018, June).
\newblock Nonlinear low-rank matrix completion for human motion recovery.
\newblock {\em IEEE Transactions on Image Processing\/}~{\em 27\/}(6),
  3011--3024.

\bibitem[\protect\citeauthoryear{Yang and Barron}{Yang and
  Barron}{1999}]{yang1999information}
Yang, Y. and A.~Barron (1999).
\newblock Information-theoretic determination of minimax rates of convergence.
\newblock {\em Annals of Statistics\/}, 1564--1599.

\bibitem[\protect\citeauthoryear{Yu, Zeng, Luo, Zhuang, He, and Shi}{Yu
  et~al.}{2013}]{yu2013embedding}
Yu, W., G.~Zeng, P.~Luo, F.~Zhuang, Q.~He, and Z.~Shi (2013).
\newblock Embedding with autoencoder regularization.
\newblock In {\em Joint European Conference on Machine Learning and Knowledge
  Discovery in Databases}, pp.\  208--223. Springer.

\end{thebibliography}

\end{document}